\documentclass{amsart}
\usepackage{amsmath}
\usepackage{amssymb}
\usepackage{amsthm}

\DeclareMathOperator{\Vol}{Vol}
\DeclareMathOperator{\dvol}{dvol}

\newcommand{\oX}{\overline{X}}

\newcommand{\of}{\overline{f}}
\newcommand{\og}{\overline{g}}

\newcommand{\orr}{\overline{r}}

\newcommand{\oDelta}{\overline{\Delta}}

\newcommand{\onabla}{\overline{\nabla}}

\newcommand{\cgamma}{\widetilde{\gamma}}

\newcommand{\hf}{\widehat{f}}
\newcommand{\hy}{\widehat{y}}
\newcommand{\hg}{\widehat{g}}

\newcommand{\hB}{\widehat{B}}

\newcommand{\hL}{\widehat{L}}

\newcommand{\hmB}{\widehat{\mathcal{B}}}
\newcommand{\hDelta}{\widehat{\Delta}}

\newcommand{\hphi}{\widehat{\phi}}

\newcommand{\lp}{\langle}
\newcommand{\rp}{\rangle}
\newcommand{\lv}{\lvert}
\newcommand{\rv}{\rvert}
\newcommand{\lV}{\lVert}
\newcommand{\rV}{\rVert}

\newcommand{\suchthat}{\mathrel{}\middle|\mathrel{}}

\newcommand{\mB}{\mathcal{B}}
\newcommand{\mC}{\mathcal{C}}

\newcommand{\mE}{\mathcal{E}}

\newcommand{\mI}{\mathcal{I}}

\newcommand{\mP}{\mathcal{P}}
\newcommand{\mQ}{\mathcal{Q}}

\newcommand{\bN}{\mathbb{N}}

\newcommand{\bR}{\mathbb{R}}

\def\sideremark#1{\ifvmode\leavevmode\fi\vadjust{\vbox to0pt{\vss
 \hbox to 0pt{\hskip\hsize\hskip1em
 \vbox{\hsize3cm\tiny\raggedright\pretolerance10000
 \noindent #1\hfill}\hss}\vbox to8pt{\vfil}\vss}}}

\newcommand{\comment}[1]{}

\newtheorem{thm}{Theorem}[section]
\newtheorem{prop}[thm]{Proposition}
\newtheorem{lem}[thm]{Lemma}
\newtheorem{cor}[thm]{Corollary}

\theoremstyle{definition}
\newtheorem{defn}[thm]{Definition}

\theoremstyle{remark}
\newtheorem{remark}[thm]{Remark}

\numberwithin{equation}{section}
\usepackage{mathtools}
\usepackage{enumitem}

\begin{document}

\title[Sharp weighted Sobolev trace inequalities]{Sharp weighted Sobolev trace inequalities and fractional powers of the Laplacian}
\author{Jeffrey S. Case}
\thanks{Partially supported by a grant from the Simons Foundation (Grant No.\ 524601)}
\address{109 McAllister Building, Department of Mathematics, Penn State University, University Park, PA 16802, USA}
\email{jscase@psu.edu}
\keywords{fractional Laplacian, sharp Sobolev trace inequality, extension}
\subjclass[2010]{Primary 35J70; Secondary 53A30}
\begin{abstract}
 We establish a family of sharp Sobolev trace inequalities involving the $W^{k,2}(\bR_+^{n+1},y^a)$-norm.  These inequalities are closely related to the realization of fractional powers of the Laplacian on $\bR^n=\partial\bR_+^{n+1}$ as generalized Dirichlet-to-Neumann operators associated to powers of the weighted Laplacian in upper half space, generalizing observations of Caffarelli--Silvestre and of Yang.
\end{abstract}
\maketitle

\section{Introduction}
\label{sec:intro}

In their seminal paper~\cite{CaffarelliSilvestre2007}, Caffarelli and Silvestre recovered the fractional Laplacian $(-\oDelta)^\gamma$, $\gamma\in(0,1)$, on $\bR^n$ as a Dirichlet-to-Neumann operator associated to the weighted Laplacian $\Delta_m:=\Delta+my^{-1}\partial_y$ in $\bR_+^{n+1}:=\bR^n\times(0,\infty)$ for $m=1-2\gamma$, where $y$ is the coordinate on $(0,\infty)$.  Specifically, if $U$ is a solution of
\begin{equation}
 \label{eqn:cs-extension}
 \begin{cases}
  \Delta_mU = 0, & \text{in $\bR_+^{n+1}$}, \\
  U = f, & \text{on $\bR^n$},
 \end{cases}
\end{equation}
then
\begin{equation}
 \label{eqn:cs-d2n}
 (-\oDelta)^\gamma f = -d_\gamma^{-1} \lim_{y\to0^+} y^m\partial_yU, \qquad d_\gamma=2^{1-2\gamma}\frac{\Gamma(1-\gamma)}{\Gamma(\gamma)} .
\end{equation}
Applying the Dirichlet principle, one deduces that
\begin{equation}
 \label{eqn:cs-energy-trace}
 \int_{\bR_+^{n+1}} \lv\nabla U\rv^2\,y^{1-2\gamma}\,dx\,dy \geq d_\gamma\oint_{\bR^n} f(-\oDelta)^\gamma f\,dx
\end{equation}
for any $U\in W^{1,2}(\bR_+^{n+1},y^{1-2\gamma})$, where $f:=U(\cdot,0)$.  Moreover, equality holds if and only if $U$ satisfies~\eqref{eqn:cs-extension}.  We may regard~\eqref{eqn:cs-energy-trace} as a functional inequality for the Sobolev trace embedding $W^{1,2}(\bR_+^{n+1},y^{1-2\gamma})\hookrightarrow H^\gamma(\bR^n)$, while the extension~\eqref{eqn:cs-extension} implies the existence of a bounded right inverse.  Combining~\eqref{eqn:cs-energy-trace} with Lieb's sharp fractional Sobolev inequality~\cite{Lieb1983} yields the sharp Sobolev trace inequality
\begin{equation}
 \label{eqn:cs-sobolev-trace}
 \int_{\bR_+^{n+1}} \lv\nabla U\rv^2\,y^{1-2\gamma}\,dx\,dy \geq \frac{\Gamma\bigl(\frac{n+2\gamma}{2}\bigr)}{\Gamma\bigl(\frac{n-2\gamma}{2}\bigr)}\Vol(S^n)^{\frac{2\gamma}{n}}d_\gamma\left(\oint_{\bR^n} \lv f\rv^{\frac{2n}{n-2\gamma}}\,dx\right)^{\frac{n-2\gamma}{n}}
\end{equation}
for any $U\in W^{1,2}(\bR_+^{n+1},y^{1+2\gamma})$, where $f:=U(\cdot,0)$.  Moreover, equality holds if and only if $U$ satisfies~\eqref{eqn:cs-extension} and there are constants $a\in\bR$ and $\varepsilon>0$ and a point $\xi\in\bR^n$ such that
\[ f(x) = a\left(\varepsilon + \lv x-\xi\rv^2\right)^{-\frac{n-2\gamma}{2}} . \]

One can also recover the fractional Laplacian $(-\oDelta)^\gamma$, $\gamma\in(0,\infty)\setminus\bN$, from the extension~\eqref{eqn:cs-extension}, though one must replace~\eqref{eqn:cs-d2n} by a limit involving additional derivatives in $y$; see~\cite{ChangGonzalez2011}.  However, this approach fails to recover a sharp Sobolev trace inequality.  Instead, R.\ Yang~\cite{ChangYang2017,Yang2013} showed that one should replace~\eqref{eqn:cs-extension} by a higher-order degenerate elliptic boundary value problem.  Roughly speaking, one can obtain $(-\oDelta)^\gamma$ through a formula similar to~\eqref{eqn:cs-d2n} by finding the solution of $\Delta_m^kU=0$ with Dirichlet boundary data, where $m:=1-2[\gamma]$ and $k:=\lfloor\gamma\rfloor+1$.  Here $\lfloor\gamma\rfloor$ is the unique integer satisfying $\lfloor\gamma\rfloor<\gamma<\lfloor\gamma\rfloor+1$ and $[\gamma]:=\gamma-\lfloor\gamma\rfloor$ is the fractional part of $\gamma$.  The Dirichlet principle for this higher-order problem gives rise to sharp Sobolev trace inequalities analogous to~\eqref{eqn:cs-energy-trace} and~\eqref{eqn:cs-sobolev-trace}, at the cost of imposing (unnecessary) boundary conditions.  We have avoided explicitly stating the sharp constants here because of a computational error in~\cite{ChangYang2017,Yang2013}; see~\eqref{eqn:intro-restricted-energy-trace} and~\eqref{eqn:intro-sobolev} below, or~\cite{Yang2019}, for the correct inequalities.

With some care, the above discussion extends to asymptotically hyperbolic manifolds.  Specifically, Graham and Zworski~\cite{GrahamZworski2003} constructed conformally covariant pseudodifferential operators on the boundary with principal symbol that of $(-\oDelta)^\gamma$ using scattering theory for the interior Laplacian.  Chang and Gonz\'alez~\cite{ChangGonzalez2011} observed (see also~\cite{CaseChang2013}) that, in the Poincar\'e upper half space model of hyperbolic space, this construction is equivalent to the Caffarelli--Silvestre extension~\cite{CaffarelliSilvestre2007}.  Chang and the author~\cite{CaseChang2013} showed that in the special case of asymptotically hyperbolic Einstein (AHE) manifolds, the Graham--Zworski operators are equivalent to particular generalized Dirichlet-to-Neumann operators associated to (weighted) GJMS operators~\cite{Case2011t,CaseChang2013,GJMS1992} in the compactification of the AHE manifold.  When restricted to the Poincar\'e upper half space, this recovers the Yang extension~\cite{ChangYang2017,Yang2013}.

Surprisingly, the higher-order fractional Laplacian $(-\oDelta)^\gamma$, $\gamma\in(1,2)\cup\{5/2\}$, can be recovered from an underdetermined (degenerate) elliptic boundary value problem~\cite{Case2015fi,Case2015b,CaseLuo2018}.  For example, if $\Delta^3U=0$ and $f=U(\cdot,0)$, then $\partial_y\Delta^2U(\cdot,0)$ is proportional to $(-\oDelta)^{5/2}f$; see~\cite{CaseLuo2018}.  A more refined version of this observation manifests as a sharp Sobolev trace inequality controlling the embedding $W^{3,2}(\bR_+^{n+1}) \hookrightarrow H^5(\bR^n) \oplus H^3(\bR^n) \oplus H^1(\bR^n)$ and an extension giving a continuous right inverse~\cite{CaseLuo2018}.  This improves the sharp Sobolev trace inequalities deduced by Yang~\cite{ChangYang2017,Yang2013} by removing the need to impose boundary conditions.  Similar results hold for sharp weighted Sobolev trace inequalities involving $W^{2,2}(\bR_+^{n+1},y^s)$, $s\in(-1,1)$; see~\cite{Case2015fi,Case2015b}.

The discussion of the previous paragraph extends to a large class of compact manifolds~\cite{Case2015fi,Case2015b,CaseLuo2018}, including all compactifications of conformally compact Einstein manifolds.  In this setting, one realizes the Graham--Zworski operators~\cite{GrahamZworski2003} as generalized Dirichlet-to-Neumann operators associated to a (weighted) GJMS operator.  The operators involved, and indeed the extension problem, are conformally covariant.  As a consequence, one readily deduces sharp Sobolev trace inequalities in the Euclidean disk from those in Euclidean upper half space.

The purpose of this article is to extend the observations described in the last two paragraphs to all $\gamma\in(0,\infty)\setminus\bN$ in the setting of Euclidean upper half space.  In the hopes of making our results more broadly accessible, we present our proofs with a minimal amount of geometric background.  We expect many of these results to extend to compactifications of Poincar\'e--Einstein manifolds.

The remainder of this introduction is devoted to explaining our main results in the special case of generalized Dirichlet-to-Neumann operators associated to the poly-Laplacian.  These recover half-integer fractional powers of the Laplacian; i.e. $(-\oDelta)^{k+\frac{1}{2}}$, $k\in\bN_0:=\{0,1,2,\dotsc\}$.  This is done both for clarity of the exposition and because we expect these cases to be of the most geometric interest.  Comments describing the general results will be given in the introduction, but only detailed in later sections.

Fix $k\in\bN_0$.  The boundary operators $B_s^{2k+1}\colon C^\infty(\overline{\bR_+^{n+1}})\to C^\infty(\bR^n)$, $0\leq s\leq 2k+1$, associated to the poly-Laplacian $(-\Delta)^{k+1}$ are defined recursively in terms of the Laplacian $\Delta$ and the derivative $\partial_y$ in $\bR_+^{n+1}$ and the induced Laplacian $\oDelta$ on $\bR^n$ as follows:
\begin{align*}
 \iota^\ast\circ\Delta^j & = \sum_{\ell=0}^j (-1)^\ell\binom{j}{\ell}\frac{(k-\ell)!}{(k-j)!}\frac{\Gamma\bigl(\frac{1}{2}+k-j-\ell\bigr)}{\Gamma\bigl(\frac{1}{2}+k-2\ell\bigr)}\oDelta^{j-\ell}B_{2\ell}^{2k+1}, \\
 \iota^\ast\circ\partial_y\Delta^j & = (-1)^{j+1}\sum_{\ell=0}^j\binom{j}{\ell}\frac{(k-\ell)!}{(k-j)!}\frac{\Gamma\bigl(\frac{3}{2}-k+2\ell\bigr)}{\Gamma\bigl(\frac{3}{2}-k+j+\ell\bigr)}\oDelta^{j-\ell}B_{2\ell+1}^{2k+1} ,
\end{align*}
where $\iota^\ast\colon C^\infty(\overline{\bR_+^{n+1}})\to C^\infty(\bR^n)$ is the restriction operator, $(\iota^\ast U)(x)=U(x,0)$.  See Definition~\ref{defn:operators} for the corresponding definitions for the boundary operators associated to the weighted GJMS operators.

These definitions are justified by three properties.  First, they are such that the associated Dirichlet form
\begin{multline}
 \label{eqn:quadratic-form}
 \mQ_{2k+1}(U,V) := \int_{\bR_+^{n+1}} U\,(-\Delta)^{k+1}V\,dx\,dy \\ + \sum_{j=0}^{\lfloor k/2\rfloor} \oint_{\bR^n} B_{2j}^{2k+1}(U)B_{2k+1-2j}^{2k+1}(V)\,dx - \sum_{j=0}^{k-\lfloor k/2\rfloor-1}\oint_{\bR^n} B_{2j+1}^{2k+1}(U)B_{2k-2j}^{2k+1}(V)\,dx
\end{multline}
is symmetric.  Denote
\[ \left\lp \Delta^{\frac{k+1}{2}}U, \Delta^{\frac{k+1}{2}}V\right\rp := \begin{cases}
  \bigl(\Delta^{(k+1)/2}U\bigr)\bigl(\Delta^{(k+1)/2}V\bigr), & \text{if $k$ is odd}, \\
  \lp \nabla\Delta^{k/2}U,\nabla\Delta^{k/2}V\rp, & \text{if $k$ is even} .
 \end{cases} \]
We in fact prove that
\[ \mQ_{2k+1}(U,V) - \int_{\bR_+^{n+1}} \left\lp \Delta^{\frac{k+1}{2}}U, \Delta^{\frac{k+1}{2}}V\right\rp\,dx\,dy \]
can be written as a symmetric boundary integral which depends only on the Dirichlet data $B_s^{2k+1}(U)$ and $B_s^{2k+1}(V)$, $0\leq s\leq k$.  See Theorem~\ref{thm:symmetry} for an explicit formula for this difference in the general case of the boundary operators associated to the weighted GJMS operators.

Second, the boundary operators associated to $(-\Delta)^{k+1}$ are covariant with respect to the group of conformal isometries of $(\bR_+^{n+1};\bR^n)$; i.e.\ the group, under composition, of maps generated by translations, rotations, and spherical inversions which fix the boundary $\bR^n=\partial\bR_+^{n+1}$.  Indeed, if $\Phi\colon\bR_+^{n+1}\to\bR_+^{n+1}$ is a conformal isometry of $(\bR_+^{n+1};\bR^n)$, then
\begin{equation}
 \label{eqn:intro-conformal-Bk}
 \Phi^\ast \left(B_{2j}^{2k+1}U\right) = (\iota^\ast J_\Phi)^{-\frac{n-1-2k+4j}{2(n+1)}} B_{2j}^{2k+1}\left( J_\Phi^{\frac{n-1-2k}{2(n+1)}}\Phi^\ast U\right)
\end{equation}
for all $0\leq j\leq 2k+1$, where $J_\Phi$ is the determinant of the Jacobian of $\Phi$.  Recall that $(-\Delta)^{k+1}$ is also conformally covariant,
\begin{equation}
 \label{eqn:intro-conformal-Lk}
 \Phi^\ast\left((-\Delta)^{k+1}U\right) = J_\Phi^{-\frac{n+3+2k}{2(n+1)}} (-\Delta)^{k+1}\left( J_\Phi^{\frac{n-1-2k}{2(n+1)}} \Phi^\ast U\right) .
\end{equation}
In particular, the right composition factors of~\eqref{eqn:intro-conformal-Bk} and~\eqref{eqn:intro-conformal-Lk} are the same and $\mQ_{2k+1}$ is conformally covariant.  See Theorem~\ref{thm:conformal_invariance} for the precise statement of conformal covariance for general $\gamma\in(0,\infty)\setminus\bN$.

Third, the boundary operators associated to $(-\Delta)^{k+1}$ are such that the generalized Dirichlet-to-Neumann operators recover the fractional Laplacians $(-\oDelta)^{\frac{1}{2}+j}$, $0\leq j\leq k$.  More precisely:

\begin{thm}
 \label{thm:intro-cs}
 Let $k\in\bN_0$ and suppose that $\Delta^{k+1}U=0$.  Then
 \begin{align}
  \label{eqn:intro-cs-even} B_{2k+1-2j}^{2k+1}(U) & = c_{j,k}(-\oDelta)^{k-2j+\frac{1}{2}}B_{2j}^{2k+1}(U), && 0\leq j\leq \lfloor k/2\rfloor , \\
  \label{eqn:intro-cs-odd} B_{2k-2j}^{2k+1}(U) & =-d_{j,k}(-\oDelta)^{k-2j-\frac{1}{2}}B_{2j+1}^{2k+1}(U), && 0\leq j\leq k-\lfloor k/2\rfloor-1 ,
 \end{align}
 where
 \begin{align}
  \label{eqn:cjk} c_{j,k} & = 2^{k-2j}\frac{(k-j)!(2k-2j+1)!!}{j!(2j-1)!!(2k-4j-1)!!(2k-4j+1)!!} , \\
  \label{eqn:djk} d_{j,k} & = 2^{k-2j-1}\frac{(k-j)!(2k-2j-1)!!}{j!(2j+1)!!(2k-4j-1)!!(2k-4j-3)!!} .
 \end{align}
\end{thm}

Here $(2j+1)!!:=1\cdot 3\cdot 5\dotsm (2j+1)$ is defined for all integers $j\geq-1$, with the convention $(-1)!!=1$.

In particular, Theorem~\ref{thm:intro-cs} states that the fractional Laplacian $(-\oDelta)^{k+\frac{1}{2}}f$ of a function $f$ can be recovered by applying $B_{2k+1}^{2k+1}$ to \emph{any} solution of $\Delta^{k+1}U$ on $\bR_+^{n+1}$ with $U(\cdot,0)=f$.  When $k=0$, solutions are unique and this recovers the Caffarelli--Silvestre extension~\cite{CaffarelliSilvestre2007}.  When $k\geq 1$, there is freedom to specify higher-order boundary data, and hence Theorem~\ref{thm:intro-cs} is more general than the Yang extension~\cite{ChangYang2017,Yang2013}.  Indeed, we readily recover the Yang extension as follows (cf.\ \cite{Yang2019}):

Given a function $f$ on $\bR^n$, let $U$ be the unique solution of
\begin{equation}
 \label{eqn:intro-extension}
 \begin{cases}
  \Delta^{k+1}U = 0, & \text{in $\bR_+^{n+1}$}, \\ 
  U = f, & \text{on $\bR^n$}, \\
  \Delta^jU = \frac{k!(2k-2j-1)!!}{2^j(k-j)!(2k-1)!!}\oDelta^jf, & \text{on $\bR^n$ for $1\leq j\leq\lfloor k/2\rfloor$}, \\
  \partial_y\Delta^jU = 0, & \text{on $\bR^n$ for $0\leq j\leq k - \lfloor k/2\rfloor-1$} .
 \end{cases}
\end{equation}
These choices ensure that $B_{2j}^{2k+1}(U)=0$ for $1\leq j\leq\lfloor k/2\rfloor$ and that $B_{2j+1}^{2k+1}(U)=0$ for $0\leq j\leq k-\lfloor k/2\rfloor-1$.  Applying Theorem~\ref{thm:intro-cs}, we deduce that
\begin{equation}
 \label{eqn:intro-cs}
 (-\oDelta)^{k+\frac{1}{2}}f = (-1)^{k+1}\frac{(2k-1)!!}{2^kk!}\partial_y\Delta^kU(\cdot,0) .
\end{equation}
This is precisely Yang's result~\cite{ChangYang2017,Yang2013}, except that the constants have been corrected so that the solution to~\eqref{eqn:intro-extension} agrees with the solution of the Poisson equation used by Graham and Zworski~\cite{GrahamZworski2003} to define $(-\oDelta)^{k+\frac{1}{2}}$ via scattering theory.

Theorem~\ref{thm:cs} below gives a more general version of Theorem~\ref{thm:intro-cs} which in particular recovers the fractional Laplacian $(-\oDelta)^\gamma$ as a generalized Dirichlet-to-Neumann operator associated to the $(\lfloor\gamma\rfloor+1)$-th power of a suitable weighted Laplacian, also generalizing the result of Yang~\cite{ChangYang2017,Yang2013}.

One reason to desire the symmetry of the quadratic form~\eqref{eqn:quadratic-form} is that it implies that many boundary value problems involving $\Delta^{k+1}$ and the boundary operators $B_{j}^{2k+1}$ are variational.  For example, if $\mB\subset\left\{B_j^{2k+1}\suchthat 0\leq j\leq 2k+1\right\}$ contains exactly $k+1$ elements, then the boundary value problem $\bigl((-\Delta)^{k+1};\mB\bigr)$ is variational if and only if $\mB$ has the property that $B_j^{2k+1}\in\mB$ if and only if $B_{2k+1-j}^{2k+1}\not\in\mB$.  Such boundary value problems are well-posed if the compatibility condition of Agmon, Douglis and Nirenberg~\cite{AgmonDouglisNirenberg1959} is also satisfied.  For example, the Dirichlet problem
\begin{equation}
 \label{eqn:intro-dirichlet-problem}
 \begin{cases}
  \Delta^{k+1}U = 0, & \text{in $\bR_+^{n+1}$}, \\
  B_{2j}^{2k+1}U = f^{(2j)}, & \text{on $\bR^n$ for $0\leq j\leq\lfloor k/2\rfloor$}, \\
  B_{2j+1}^{2k+1}U = \phi^{(2j)}, & \text{on $\bR^n$ for $0\leq j\leq k-\lfloor k/2\rfloor-1$}
 \end{cases}
\end{equation}
is well-posed.  Here we are specifying $\mB=\left\{B_j^{2k+1}\suchthat 0\leq j\leq k\right\}$; we have written~\eqref{eqn:intro-dirichlet-problem} in this somewhat strange way to highlight a distinction between the ``even'' and ``odd'' boundary operators which is more pronounced for the Dirichlet problem associated to the fractional Laplacian $(-\oDelta)^\gamma$, $\gamma\in(0,\infty)\setminus(\frac{1}{2}\bN)$; cf.\ \eqref{eqn:bvp}.

The Dirichlet problem~\eqref{eqn:intro-dirichlet-problem} can be solved by minimizing the energy
\[ \mE_{2k+1}(U) := \mQ_{2k+1}(U,U) \]
among all functions $U$ with prescribed Dirichlet data.  Combining this with Theorem~\ref{thm:intro-cs} yields the following sharp Sobolev trace inequality:

\begin{thm}
 \label{thm:intro-energy-trace}
 Let $k\in\bN_0$.  Given any function $U$ on $\bR_+^{n+1}$, it holds that
 \begin{multline}
  \label{eqn:intro-energy-trace}
  \mE_{2k+1}(U) \geq \sum_{j=0}^{\lfloor k/2\rfloor} c_{j,k}\oint_{\bR^n} f^{(2j)}(-\oDelta)^{k-2j+\frac{1}{2}}f^{(2j)}\,dx \\ + \sum_{j=0}^{k-\lfloor k/2\rfloor-1} d_{j,k}\oint_{\bR^n} \phi^{(2j)}(-\oDelta)^{k-2j-\frac{1}{2}}\phi^{(2j)}\,dx ,
 \end{multline}
 where $f^{(2j)}:=B_{2j}^{2k+1}(U)$ and $\phi^{(2j)}:=B_{2j+1}^{2k+1}(U)$, and the constants $c_{j,k}$ and $d_{j,k}$ are given by~\eqref{eqn:cjk} and~\eqref{eqn:djk}, respectively.  Moreover, equality holds if and only if $U$ is the unique solution of~\eqref{eqn:intro-dirichlet-problem}.
\end{thm}

We refer to~\eqref{eqn:intro-energy-trace} as a sharp Sobolev trace inequality because it easily establishes the (well-known) embedding
\[ W^{k+1,2}(\bR_+^{n+1}) \hookrightarrow \bigoplus_{j=0}^{k} H^{j+\frac{1}{2}}(\bR^n) \]
as well as the existence of a continuous right inverse.  The analogue of Theorem~\ref{thm:intro-energy-trace} involving fractional Laplacians of general order is stated as Corollary~\ref{cor:energy-trace} below.

Consider for the moment the special class of functions
\[ \mC_+ := \left\{ U \in W^{k+1,2}(\bR_+^{n+1}) \suchthat B_{j}^{2k+1}(U)=0, 1\leq j\leq k \right\} ; \]
note that $U$ satisfies the boundary conditions of~\eqref{eqn:intro-extension} if and only if $U\in\mC_+$ and $U(\cdot,0)=f$.  Theorem~\ref{thm:intro-energy-trace} implies that
\begin{equation}
 \label{eqn:intro-restricted-energy-trace}
 \mE_{2k+1}(U) \geq \frac{2^kk!}{(2k-1)!!}\oint_{\bR^n} f(-\oDelta)^{k+\frac{1}{2}}f\,dx
\end{equation}
for any $U\in\mC_+$, where $f:=U(\cdot,0)$.  Moreover, equality holds if and only if $U$ solves~\eqref{eqn:intro-extension}.  This inequality provides a starting point for many of the sharp Sobolev trace inequalities on manifolds recently considered in the literature; e.g.\ \cite{AcheChang2015,Case2015b,CaseLuo2018,NgoNguyenPhan2018}.

Combining Theorem~\ref{thm:intro-energy-trace} with the sharp fractional Sobolev inequalities~\cite{Lieb1983} yields the following more typical formulation of a sharp Sobolev trace inequality.

\begin{thm}
 \label{thm:intro-trace}
 Let $k\in\bN_0$ and suppose that $n>2k+1$.  Given any function $U$ on $\bR_+^{n+1}$, it holds that
 \begin{multline}
  \label{eqn:intro-trace}
  \mE_{2k+1}(U) \geq \sum_{j=0}^{\lfloor k/2\rfloor} c_{j,k}\frac{\Gamma(\frac{n+2k-4j+1}{2})}{\Gamma(\frac{n-2k+4j-1}{2})}\Vol(S^n)^{\frac{2k-4j+1}{n}}\lV f^{(2j)}\rV_{\frac{2n}{n-2k+4j-1}}^2 \\ + \sum_{j=0}^{k-\lfloor k/2\rfloor-1} d_{j,k}\frac{\Gamma(\frac{n+2k-4j-1}{2})}{\Gamma(\frac{n-2k+4j+1}{2})}\Vol(S^n)^{\frac{2k-4j-1}{n}}\lV \phi^{(2j)}\rV_{\frac{2n}{n-2k+4j+1}}^2 ,
 \end{multline}
 where $f^{(2j)}:=B_{2j}^{2k+1}(U)$ and $\phi^{(2j)}:=B_{2j+1}^{2k+1}(U)$, and the constants $c_{j,k}$ and $d_{j,k}$ are given by~\eqref{eqn:cjk} and~\eqref{eqn:djk}, respectively.  Moreover, equality holds if and only if $U$ satisfies~\eqref{eqn:intro-dirichlet-problem} and there are constants $a_j,b_\ell\in\bR$ and $\varepsilon_j,\epsilon_\ell\in(0,\infty)$ and points $\xi_j,\zeta_\ell\in\bR^n$ such that
 \begin{align*}
  f^{(2j)}(x) & = a_j\left(\varepsilon_j + \lv x-\xi_j\rv^2\right)^{-\frac{n-2k+4j-1}{2}}, && 0\leq j\leq\lfloor k/2\rfloor, \\
  \phi^{(2\ell)}(x) & = b_\ell\left(\epsilon_\ell + \lv x-\zeta_\ell\rv^2\right)^{-\frac{n-2k+4j+1}{2}}, && 0\leq\ell\leq k-\lfloor k/2\rfloor - 1 .
 \end{align*}
\end{thm}

The special case $k=0$ was proven by Escobar~\cite{Escobar1988}; the special case $k=1$ was proven by the author~\cite{Case2015b}; and the special case $k=2$ by Luo and the author~\cite{CaseLuo2018}.  As a special case of Theorem~\ref{thm:intro-trace}, we deduce that
\begin{equation}
 \label{eqn:intro-sobolev}
 \mE_{2k+1}(U) \geq \frac{2^kk!}{(2k-1)!!}\frac{\Gamma(\frac{n+2k+1}{2})}{\Gamma(\frac{n-2k-1}{2})}\Vol(S^n)^{\frac{2k+1}{n}}\lV U(\cdot,0)\rV_{\frac{2n}{n-2k-1}}
\end{equation}
with equality if and only if there are constants $a\in\bR$ and $\varepsilon\in(0,\infty)$ and a point $\xi\in\bR^n$ such that $U$ is the extension by~\eqref{eqn:intro-extension} of
\[ f(x) = a\left(\varepsilon + \lv x-\xi\rv^2\right)^{-\frac{n-2k-1}{2}} . \]
Q.\ Yang~\cite{Yang2019} recently gave a similar proof of~\eqref{eqn:intro-sobolev} which also leads to the corresponding sharp inequality in the Euclidean ball; see also~\cite{AcheChang2015,NgoNguyenPhan2018} for the low-order cases.  We expect the aforementioned conformal covariance of the boundary operators $B_{2j}^{2k+1}$ to lead to the analogue of~\eqref{eqn:intro-trace} in the Euclidean ball.

Theorem~\ref{thm:sobolev-trace} below gives the analogue of Theorem~\ref{thm:intro-trace} which applies to functions in weighted Sobolev spaces.

The replacement of the sharp Sobolev inequality~\eqref{eqn:intro-sobolev} in the critical dimension $n=2k+1$ is the following sharp Lebedev--Milin-type inequality. 

\begin{thm}
 \label{thm:intro-lebedev-milin}
 Let $k\in\bN_0$ and set $n:=2k+1$.  Given any function $U\in\mC_+(\bR_+^{n+1})$, it holds that
 \[ \mE_{n}(U) \geq 2^{2k+1}(2k+1)(k!)^2\Vol(S^{n})\ln\oint_{\bR^{n}} e^{f-\of}\,d\mu , \]
 where $f:=U(\cdot,0)$ and $\of$ is the average of $f$ with respect to $d\mu:=\frac{1}{\Vol(S^n)}\bigl(\frac{2}{1+x^2}\bigr)^{n}dx$.  Moreover, equality holds if and only if $U$ is the solution of~\eqref{eqn:intro-extension} and there are constants $a\in\bR$ and $\varepsilon\in(0,\infty)$ and a point $\xi\in\bR^{n}$ such that
 \[ f(x) = a - \ln\frac{\varepsilon+\lv x-\xi\rv^2}{1+\lv x\rv^2} . \] 
\end{thm}

The special case $k=0$ was proven by Osgood, Phillips and Sarnak~\cite{OsgoodPhillipsSarnak1988}; the special case $k=1$ was proven by Ache--Chang~\cite{AcheChang2015} and the author~\cite{Case2015b}; the special case $k=2$ was proven by Luo and the author~\cite{CaseLuo2018}.  In order to avoid unnecessary redundancies, we have opted to state our sharp Lebedev--Milin inequality for $\mC_+$ only; for general functions, the sharp inequality will also include $L^p$-norms on the boundary data $f^{(2j)}$, $1\leq j\leq\lfloor k/2\rfloor$, and $\phi^{(2j)}$, $0\leq j\leq k-\lfloor k/2\rfloor-1$, as in Theorem~\ref{thm:intro-trace}.

As previously noted, we expect many of our results extend to the compactification $(\oX,\og)$ of a Poincar\'e--Einstein manifolds $(X^{n+1},g_+)$.  Specifically, we expect that there are boundary operators associated to the (weighted) GJMS operators on any Riemannian manifold with boundary, as is already known for the (weighted) conformal Laplacian~\cite{Case2015fi,Cherrier1984,Escobar1988}, the (weighted) Paneitz operator~\cite{Case2015fi,Case2015b}, and, under a mild assumption on the boundary, the sixth-order GJMS operator~\cite{CaseLuo2018}.  However, we only expect a relationship to the fractional GJMS operators as in Theorem~\ref{thm:intro-cs} when there is a Poincar\'e--Einstein metric $g_+$ in the interior, as our argument relies heavily on the factorization of the (weighted) GJMS operators at Einstein metrics~\cite{CaseChang2013,FeffermanGraham2012,Gover2006q}.  In the Poincar\'e--Einstein setting, one can use the recursive formula for the formal solution of the scattering equation~\cite{GrahamZworski2003} to find candidate boundary operators in terms of geodesic compactifications, in the sense that the appropriate generalization of Theorem~\ref{thm:intro-cs} holds.  However, these recursive formulae are not completely explicit (cf.\ \cite{Juhl2013}), making it difficult to prove that the induced Dirichlet form $\mQ$ is symmetric or better understand the geometric content of the boundary operators. 

This article is organized as follows:

In Section~\ref{sec:scattering} we recall the identification of fractional powers of the Laplacian on Euclidean space via scattering theory for the hyperbolic Laplacian in Poincar\'e upper half space~\cite{GrahamZworski2003} and give a direct relationship between powers of the weighted Laplacian in upper half space and weighted GJMS operators in hyperbolic space (cf.\ \cite{CaseChang2013}).

In Section~\ref{sec:boundary} we introduce the boundary operators associated to powers of the weighted Laplacian and study their role in recovering fractional powers of the Laplacian as generalized Dirichlet-to-Neumann operators.  The key results here link our boundary operators to the asymptotics of solutions to a Poisson equation relevant to scattering theory~\cite{GrahamZworski2003} and show that the Dirichlet form determined by our boundary operators is symmetric.

In Section~\ref{sec:conformal} we prove that the boundary operators associated to powers of the weighted Laplacian are conformally covariant with respect to the conformal group of $(\bR_+^{n+1};\bR^n)$.

In Section~\ref{sec:cs} we prove the existence and uniqueness of solutions to the weighted analogue of~\eqref{eqn:intro-dirichlet-problem}, which is a higher-order degenerate elliptic boundary value problem.  This enables us to prove the main result of this section, Theorem~\ref{thm:cs}, which asserts the analogue of Theorem~\ref{thm:intro-cs} for all fractional powers of the Laplacian.

In Section~\ref{sec:trace} we prove various sharp trace inequalities.  There are two main results in this section.  First, Theorem~\ref{thm:dirichlet-principle} asserts a Dirichlet principle for solutions of the Dirichlet problem considered in Theorem~\ref{thm:bvp}.  As a result, we obtain in Corollary~\ref{cor:energy-trace} the analogue of Theorem~\ref{thm:intro-energy-trace} for all fractional powers of the Laplacian.  Second, Theorem~\ref{thm:sobolev-trace} asserts a sharp Sobolev trace inequality which passes through all fractional powers of the Laplacian, generalizing Theorem~\ref{thm:intro-trace}.  The same argument in the critical dimension establishes Theorem~\ref{thm:intro-lebedev-milin}.

\section{Fractional powers of the Laplacian via scattering theory}
\label{sec:scattering}

Conformally covariant pseudodifferential operators with principal symbol that of a fractional Laplacian were defined by Graham and Zworski~\cite{GrahamZworski2003} using scattering theory for the Laplacian of an asymptotically hyperbolic manifold.  In the special case of Euclidean space, these operators are equivalent to the definition of the fractional Laplacian via Fourier transform; see, for example, \cite{ChangGonzalez2011}.  Given our expectation that many of our results generalize to boundaries of AHE manifolds (cf.\ \cite{Case2015fi,Case2015b,CaseLuo2018}), we study fractional powers of the Laplacian via scattering theory.  Here we summarize this construction in the special case of Euclidean space as the conformal boundary of the Poincar\'e upper half space model of hyperbolic space.

Let $(x,y)\in\bR^n\times(0,\infty)=:\bR_+^{n+1}$ denote coordinates in upper half space, regard $\bR^n=\bR^n\times\{0\}$ as the boundary of $\bR_+^{n+1}$, and let $g_+:=y^{-2}(dx^2+dy^2)$ denote the hyperbolic metric on $\bR_+^{n+1}$.  Given $\gamma\in(0,\infty)\setminus\bN$, for each $f\in C^\infty(\bR^n)\cap H^\gamma(\bR^n)$, there is a unique solution $\mP\bigl(\frac{n}{2}+\gamma\bigr)(f)$ of the Poisson equation
\begin{equation}
 \label{eqn:poisson}
 \Delta_{g_+}V + \left(\frac{n^2}{4}-\gamma^2\right)V = 0 
\end{equation}
with $y^{-\frac{n-2\gamma}{2}}\mP\bigl(\frac{n}{2}+\gamma\bigr)(f)(\cdot,y)\to f(\cdot)$ as $y\to0^+$.  See~\cite{DelaTorreGonzalezHyderMartinazzi2018,Yang2019} for a Poisson kernel for~\eqref{eqn:poisson}.  For our purposes, it suffices to know that there are functions $F,G\in C^\infty(\overline{\bR_+^{n+1}})$ such that
\begin{equation}
 \label{eqn:poisson_asymptotics}
 \mP\left(\frac{n}{2}+\gamma\right)(f) = y^{\frac{n-2\gamma}{2}}F + y^{\frac{n+2\gamma}{2}}G
\end{equation}
and $F(\cdot,0)=f$.  Set $S\bigl(\frac{n}{2}+\gamma\bigr)(f):=G(\cdot,0)$.  The \emph{fractional GJMS operator $P_{2\gamma}$ of order $2\gamma$} is
\begin{equation}
 \label{eqn:fractional_gjms_definition}
 P_{2\gamma} := 2^{2\gamma}\frac{\Gamma(\gamma)}{\Gamma(-\gamma)}S\left(\frac{n}{2}+\gamma\right) .
\end{equation}
Graham and Zworski~\cite{GrahamZworski2003} showed that $P_{2\gamma}=(-\Delta)^\gamma$.

The function $F$ in~\eqref{eqn:poisson_asymptotics} is determined modulo $O(y^\infty)$ by $f$ by finding the Taylor series solution to
\[ \left(\Delta_{g_+} + \frac{n^2}{4}-\gamma^2\right)\left(y^{\frac{n-2\gamma}{2}}F\right) = O(y^\infty) \]
with $F(\cdot,0)=f$.  Similarly, the function $G$ is determined modulo $O(y^\infty)$ by $P_{2\gamma}f$ by finding the Taylor series solution to
\[ \left(\Delta_{g_+} + \frac{n^2}{4}-\gamma^2\right)\left(y^{\frac{n+2\gamma}{2}}G\right) = O(y^\infty) \]
with $G(\cdot,0)=S\bigl(\frac{n}{2}+\gamma\bigr)(f)$.  For this reason, we want to understand the formal solutions of $\Delta_{g_+}V+s(n-s)V=0$ when $s=\frac{n}{2}\pm\gamma$, $\gamma\in(0,\infty)\setminus\bN$.

\begin{prop}
 \label{prop:formal_solutions}
 Let $(\bR_+^{n+1},g_+)$ denote $(n+1)$-dimensional hyperbolic space with boundary $(\bR^n,dx^2)$ and let $s\in\bR$.  Given $u\in C^\infty(\bR^n)$, set
 \begin{equation}
  \label{eqn:formal_F}
  U(x,y) := y^s\sum_{j=0}^\infty \frac{\Gamma\bigl(s-\frac{n}{2}+1\bigr)}{2^{2j}j!\Gamma\bigl(j+s-\frac{n}{2}+1\bigr)}y^{2j}(-\oDelta)^ju(x) ,
 \end{equation}
 where $\oDelta=\sum_{i=1}^n\partial_{x_i}^2$.  Then
 \begin{equation}
  \label{eqn:formal_poisson}
  \Delta_{g_+}U + s(n-s)U = O(y^\infty) .
 \end{equation}
\end{prop}

\begin{proof}
 A straightforward computation verifies that
 \begin{equation}
  \label{eqn:hyperbolic-laplacian-expression}
  \Delta_{g_+} = y^2\partial_y^2 - (n-1)y\partial_y + y^2\oDelta .
 \end{equation}
 Let $W\in C^\infty(\bR_+^{n+1})$ be such that $\partial_yW=0$.  For any $p\in\bR$ it holds that
 \begin{equation}
  \label{eqn:induction}
  \Delta_{g_+}\left(y^pW\right) = y^p\left( p(p-n)W + y^2\oDelta W \right) .
 \end{equation}
 It follows immediately from~\eqref{eqn:induction} that if $\{c_j\}_{j=0}^\infty$ is the sequence such that
 \begin{equation}
  \label{eqn:recursion}
  c_0=1, \qquad c_j = \frac{1}{2j(2s+2j-n)}c_{j-1} \text{ for $j\geq1$},
 \end{equation}
 then
 \[ U(x,y) = y^s\sum_{j=0}^\infty c_jy^{2j}(-\oDelta)^ju(x) \]
 satisfies~\eqref{eqn:formal_poisson}.  We readily check that the solution to~\eqref{eqn:recursion} is
 \[ c_j = 2^{-2j}\frac{\Gamma\bigl(s-\frac{n}{2}+1\bigr)}{j!\Gamma\bigl(j+s-\frac{n}{2}+1\bigr)} . \qedhere \]
\end{proof}

There are two ways to study the fractional Laplacian via an extension.  The first approach is to identify solutions of~\eqref{eqn:poisson} as elements of the kernel of a second-order weighted Laplacian on Euclidean space (cf.\ \cite{ChangGonzalez2011}):

\begin{lem}
 \label{lem:chang_gonzalez}
 Let $\gamma\in(0,\infty)\setminus\bN$.  Set $m_0:=1-2\gamma$ and define
 \begin{equation}
  \label{eqn:weighted_laplacian}
  \Delta_{m_0} := \Delta + m_0y^{-1}\partial_y .
 \end{equation}
 Then
 \begin{equation}
  \label{eqn:chang_gonzalez}
  \Delta_{m_0} \circ y^{-\frac{n-2\gamma}{2}} \circ \mP\left(\frac{n}{2}+\gamma\right) = 0 .
 \end{equation}
\end{lem}

\begin{proof}
 A direct computation using~\eqref{eqn:hyperbolic-laplacian-expression} shows that
 \begin{equation}
  \label{eqn:conf-lapl-transformation}
  \Delta_{m_0}\left(y^{-\frac{n-2\gamma}{2}}U\right) = y^{-\frac{n-2\gamma+4}{2}}\left(\Delta_{g_+}+\frac{n^2}{4}-\gamma^2\right)U
 \end{equation}
 for all $U\in C^\infty(\bR_+^{n+1})$.  The conclusion follows from the definition of the Poisson operator $\mP\bigl(\frac{n}{2}+\gamma\bigr)$.
\end{proof}

The operator~\eqref{eqn:weighted_laplacian} is formally self-adjoint with respect to the measure $y^{m_0}\,dx\,dy$ on $\bR_+^{n+1}$.  However, this measure is only locally finite in $\overline{\bR_+^{n+1}}$ when $\gamma\in(0,1)$, precluding us from using Lemma~\ref{lem:chang_gonzalez} to obtain energy estimates for the fractional Laplacian in terms of interior energy estimates in general.

The second approach to studying the fractional Laplacian via extensions is to identify solutions of~\eqref{eqn:poisson} as elements of the kernel of powers of a weighted Laplacian (cf.\ \cite{ChangYang2017,Yang2013}).  This can be done as follows:

First, let $\gamma\in(0,\infty)\setminus\bN$.  Set $k:=\lfloor\gamma\rfloor+1$ and $m:=1-2[\gamma]$.  The \emph{weighted poly-Laplacian determined by $\gamma$} is
\begin{equation}
 \label{eqn:weighted_polylaplacian}
 L_{2k} := \Delta_m^k,
\end{equation}
where $\Delta_m$ is given by~\eqref{eqn:weighted_laplacian}.

Second, let $\gamma\in(0,\infty)\setminus\bN$.  Set $s:=\frac{n}{2}+\gamma$ and
\[ D_s := \Delta_{g_+} + s(n-s) . \]
The \emph{hyperbolic poly-Poisson operator determined by $\gamma$} is
\begin{equation}
 \label{eqn:hyperbolic_polyharmonic_operator}
 L_{2k}^+ := \prod_{j=0}^{k-1}D_{s-2j} .
\end{equation}

The operators $L_{2k}$ and $L_{2k}^+$ are closely related.  Succinctly, they are the weighted GJMS operators of order $2k$ determined by the smooth metric measure spaces $(\bR_+^{n+1},y^2g_+,y^m\,dx\,dy,m-1)$ and $(\bR_+^{n+1},g_+,\dvol_{g_+},m-1)$, respectively, and hence, by conformal covariance, these operators are the same on densities (see~\cite{CaseChang2013}).  The following two lemmas capture the essential features of this relationship as needed to study sharp Sobolev trace inequalities in upper half space.

First, the factorization~\eqref{eqn:hyperbolic_polyharmonic_operator} is the factorization of weighted GJMS operators~\cite{CaseChang2013}.  However, to understand the boundary operators, it is more useful to write the factorization in a different form.

\begin{lem}
 \label{lem:factorization}
 Let $\gamma\in(0,\infty)\setminus\bN$ and consider the hyperbolic poly-Poisson operator~\eqref{eqn:hyperbolic_polyharmonic_operator}.  Then
 \begin{multline}
  \label{eqn:factorization}
  L_{2k}^+ = \left\{ \prod_{j=0}^{\lfloor\gamma/2\rfloor} \Delta_{g_+} + \left(\frac{n}{2}+\gamma-2j\right)\left(\frac{n}{2}-\gamma+2j\right) \right\} \\ \times \left\{ \prod_{j=0}^{\lfloor\gamma\rfloor-\lfloor\gamma/2\rfloor-1} \Delta_{g_+} + \left( \frac{n}{2}+\lfloor\gamma\rfloor - [\gamma] - 2j \right)\left( \frac{n}{2}-\lfloor\gamma\rfloor+[\gamma]+2j \right) \right\} ,
 \end{multline}
 with the convention that the empty product equals one.
\end{lem}

\begin{proof}
 Let $\ell=\lfloor \gamma/2\rfloor$.  Separating~\eqref{eqn:hyperbolic_polyharmonic_operator} into terms with $s-2j>n$ and $s-2j<n$, we compute that
 \begin{align*}
  L_{2k}^+ & = \left\{ \prod_{j=0}^{\ell} D_{s-2j} \right\}\left\{ \prod_{j=\ell+1}^{k-1} D_{s-2j} \right\} \\
  & = \left\{ \prod_{j=0}^{\ell} D_{s-2j} \right\}\left\{ \prod_{j=0}^{k-\ell-2} D_{s-2k+2+2j}\right\} ,
 \end{align*}
 where the second equality follows by reindexing.  Rewriting the latter expression in terms of $\gamma$ yields the desired result.
\end{proof}

Second, elements of the kernel of $D_{s-2j}$ are also in the kernel of $L_{2k}$ when weighted against a suitable power of $y$; this power is precisely the one required by conformal covariance~\cite{CaseChang2013}.  To prove this without appealing to conformal covariance requires the following lemma.

\begin{lem}
 \label{lem:commutator}
 Fix $m\in\bR$ and denote $\Delta_m:=\Delta+my^{-1}\partial_y$.  Then for any $k\in\bN$ it holds that
 \[ \Delta_{m+2k}\Delta_m^{k-1}\Delta_{m-2k} = \Delta_m^{k+1} . \]
\end{lem}

\begin{proof}
 First observe the commutator identity
 \begin{equation}
  \label{eqn:basic-commutator}
  [\Delta_m,y^{-1}\partial_y] = -2(y^{-1}\partial_y)^2 . 
 \end{equation}
 On the other hand, it follows from the definition of $\Delta_m$ that
 \begin{equation}
  \label{eqn:pre-commutator}
  \Delta_{m+2k}\Delta_m^{k-1}\Delta_{m-2k} = \Delta_m^{k+1} - 2k[\Delta_m^k,y^{-1}\partial_y] - 4k^2y^{-1}\partial_y\Delta_m^{k-1}y^{-1}\partial_y . 
 \end{equation}
 It thus suffices to show that $[\Delta_m^k,y^{-1}\partial_y] = -2ky^{-1}\partial_y\Delta_m^{k-1}y^{-1}\partial_y$ for all $k \in \bN$.
 The proof is by induction:
 
 Denote $Y := y^{-1}\partial_y$.
 Suppose $k \in \bN$ is such that $[\Delta_m^k,Y] = -2kY \Delta_m^{k-1} Y$.
 Then
 \begin{align*}
  [\Delta_m^{k+1},Y] & = \Delta_m[\Delta_m^k,Y] + [\Delta_m,Y]\Delta_m^k \\
   & = -2k\Delta_m Y \Delta_m^{k-1} Y - 2Y^2\Delta_m^k \\
   & = -2(k+1)Y\Delta_m^kY - 2k[\Delta_m,Y]\Delta_m^{k-1}Y - 2Y[Y,\Delta^m] \\
   & = -2(k+1)Y\Delta_m^kY . \qedhere
 \end{align*}
\end{proof}

We now prove that elements of the kernel of $D_{s-2j}$ are also in the kernel of $L_{2k}$ when weighted against a suitable power of $y$.

\begin{lem}
 \label{lem:kernel}
 Let $\gamma\in(0,\infty)\setminus\bN$.  Set $k:=\lfloor\gamma\rfloor+1$ and $m:=1-2[\gamma]$.  Denote
 \[ \mI^{2\gamma} := \left\{ \gamma-2j \suchthat 0\leq j\leq \lfloor\gamma/2\rfloor \right\} \cup \left\{ \lfloor\gamma\rfloor - [\gamma] - 2j \suchthat 0 \leq j\leq \lfloor\gamma\rfloor - \lfloor\gamma/2\rfloor - 1 \right\} \]
 For each $\cgamma\in\mI^{2\gamma}$ it holds that
 \[ L_{2k}\circ y^{-\frac{n-2\gamma}{2}} \circ \mP\left(\frac{n}{2}+\cgamma\right) = 0 . \]
\end{lem}

\begin{proof}
 It is clear from Lemma~\ref{lem:factorization} that
 \begin{equation}
  \label{eqn:hyperbolic-annihilator}
  L_{2k}^+ \circ \mP\left(\frac{n}{2}+\cgamma\right) = 0
 \end{equation}
 for each $\cgamma\in\mI^{2\gamma}$.  Thus it suffices to prove that
 \begin{equation}
  \label{eqn:gjms-transformation}
  L_{2k}\circ y^{-\frac{n-2\gamma}{2}} = y^{-\frac{n-2\gamma+4k}{2}}L_{2k}^+ .
 \end{equation}
 To that end, observe that~\eqref{eqn:hyperbolic_polyharmonic_operator} and a repeated application of~\eqref{eqn:conf-lapl-transformation} implies that
 \begin{equation}
  \label{eqn:pretransformation}
  y^{-\frac{n-2\gamma+4k}{2}}L_{2k}^+ = \left(\prod_{j=0}^{k-1}\Delta_{m-2k+4j+2}\right)\circ y^{-\frac{n-2\gamma}{2}} ,
 \end{equation}
 where
 \[ \prod_{j=0}^{k-1}\Delta_{m-2k+4j+2} := \Delta_{m+2k-2}\circ\Delta_{m+2k-6}\circ\dotsb\circ\Delta_{m-2k+2} . \]
 An induction argument using Lemma~\ref{lem:commutator} implies that
 \[ \prod_{j=0}^{k-1} \Delta_{m-2k+4j+2} = \Delta_m^k . \]
 Inserting this into~\eqref{eqn:pretransformation} yields~\eqref{eqn:gjms-transformation}
\end{proof}

\section{Boundary operators in upper half space}
\label{sec:boundary}

In this section we introduce the boundary operators associated to the weighted poly-Laplacian $L_{2k}$ determined by $\gamma\in(0,\infty)\setminus\bN$.  By Lemma~\ref{lem:kernel}, the kernel of $L_{2k}$ contains solutions of the Poisson equation~\eqref{eqn:poisson} for any $\cgamma\in\mI^{2\gamma}$.  Our boundary operators are designed to pick out the functions $F(\cdot,0)$ and $G(\cdot,0)$ of solutions to~\eqref{eqn:poisson}.  They also give rise to formally self-adjoint boundary value problems; in fact, Theorem~\ref{thm:symmetry} gives a stronger result.  To that end, it is convenient to introduce the space
\[ \mC^{2\gamma} = C_{\mathrm{even}}^\infty(\overline{\bR_+^{n+1}}) + y^{2[\gamma]}C_{\mathrm{even}}^{\infty}(\overline{\bR_+^{n+1}}) \]
associated to a given $\gamma\in(0,\infty)\setminus\bN$, where $C_{\mathrm{even}}^{\infty}(\overline{\bR_+^{n+1}})$ denotes the space of smooth functions on $\overline{\bR_+^{n+1}}$ whose Taylor series expansions in $y$ at $y=0$ contain only even terms.  Note that 
\begin{enumerate}
 \item if $\gamma\in\frac{1}{2}+\bN_0$, then $\mC^{2\gamma}=C^\infty(\overline{\bR_+^{n+1}})$; and
 \item for any $\cgamma\in\mI^{2\gamma}$, it holds that $\mP\bigl(\frac{n}{2}+\cgamma\bigr)\colon C^\infty(\bR^n)\cap H^{\cgamma}(\bR^n) \to \mC^{2\gamma}$.
\end{enumerate}
The second point means that the space $\mC^\gamma$ is well-suited to studying all of the scattering problems formed from the factors of the hyperbolic poly-Poisson operator~\eqref{eqn:hyperbolic_polyharmonic_operator}.  A similar definition of $\mC^{2\gamma}$ should be made for compactifications of asymptotically hyperbolic manifolds; see~\cite{Case2015fi} for a discussion of the case $\gamma\in(0,1)\cup(1,2)$.

As in the introduction, let $\iota^\ast\colon\mC^{2\gamma}\to C^\infty(\bR^n)$ denote the restriction operator, $(\iota^\ast U)(x)=U(x,0)$.  Our boundary operators are elements of the set
\begin{multline}
 \label{eqn:boundary_operators}
 \mB^{2\gamma} := \left\{ B_{2j}^{2\gamma}\colon\mC^{2\gamma}\to C^\infty(\bR^n) \suchthat 0\leq j\leq\lfloor\gamma\rfloor \right\} \\ \cup \left\{ B_{2[\gamma]+2j}^{2\gamma}\colon\mC^{2\gamma}\to C^\infty(\bR^n) \suchthat 0\leq j\leq \lfloor\gamma\rfloor \right\}
\end{multline}
defined as follows:

\begin{defn}
 \label{defn:operators}
 Let $\gamma\in(0,\infty)\setminus\bN$ and set $m:=1-2[\gamma]$.  Given $0\leq j\leq\lfloor\gamma\rfloor$, we define $B_{2j}^{2\gamma}\colon\mC^{2\gamma}\to C^\infty(\bR^n)$ recursively by  
 \[ B_{2j}^{2\gamma} = (-1)^j\iota^\ast\circ T^j - \sum_{\ell=1}^j \binom{j}{\ell}\frac{\Gamma(1+j-[\gamma])\Gamma(1+2j-2\ell-\gamma)}{\Gamma(1+j-\ell-[\gamma])\Gamma(1+2j-\ell-\gamma)}\oDelta^\ell B_{2j-2\ell}^{2\gamma} , \]
 where $T:=\partial_y^2+my^{-1}\partial_y$ and the empty sum equals zero by convention.  Likewise, given $0\leq j\leq\lfloor\gamma\rfloor$, we define $B_{2[\gamma]+2j}^{2\gamma}\colon\mC^{2\gamma}\to C^\infty(\bR^n)$ recursively by
 \begin{multline*}
  B_{2[\gamma]+2j}^{2\gamma} = (-1)^{j+1}\iota^\ast\circ y^m\partial_yT^j \\ - \sum_{\ell=1}^j \binom{j}{\ell}\frac{\Gamma(1+j+[\gamma])\Gamma(1+2j-2\ell-\lfloor\gamma\rfloor+[\gamma])}{\Gamma(1+j-\ell+[\gamma])\Gamma(1+2j-\ell-\lfloor\gamma\rfloor+[\gamma])}\oDelta^\ell B_{2[\gamma]+2j-2\ell}^{2\gamma} .
 \end{multline*}
\end{defn}

It is straightforward to show that any operator $B_{2\alpha}^{2\gamma}\in\mB^{2\gamma}$ is a homogeneous differential operator of degree $2\alpha$ which can be written as a polynomial in $\oDelta$, $T:=\partial_y^2+my^{-1}\partial_y$, and $y^m\partial_y$; in fact, it is a polynomial in $\oDelta$ and $T$ when $\alpha\in\bN_0$ and the composition of $y^m\partial_y$ with such a polynomial when $\alpha\not\in\bN_0$.  Moreover, the leading coefficient --- in the sense that it corresponds to the term in which $\oDelta$ does not appear --- is $\pm1$.  These properties are relevant for computing the energy associated to $L_{2k}$ and the boundary operators $\mB^{2\gamma}$; see Section~\ref{sec:trace} for details.

The first goal of this section is to show that the boundary operators $\mB^{2\gamma}$ are relevant for picking out the Dirichlet data $F(\cdot,0)$ and the Neumann data $G(\cdot,0)$ of solutions of the Poisson equation $\mP\bigl(\frac{n}{2}+\cgamma\bigr)$ for $\cgamma\in\mI^{2\gamma}$.  This is accomplished by the following two propositions.

\begin{prop}
 \label{prop:reln-to-scattering}
 Let $\gamma\in(0,\infty)\setminus\bN$.  Given $0\leq j\leq\lfloor\gamma/2\rfloor$, set $\cgamma:=\gamma-2j$.  Let $V=y^{-\frac{n-2\gamma}{2}}\mP\bigl(\frac{n}{2}+\cgamma\bigr)f$ for some $f\in C^\infty(\bR^n)\cap H^{\cgamma}(\bR^n)$.  It holds that
 \begin{align}
  \label{eqn:reln-to-scattering-fn} B_{2j}^{2\gamma}(V) & = (-1)^j2^{2j}j!\frac{\Gamma(1+j-[\gamma])}{\Gamma(1-[\gamma])}f, \\
  \label{eqn:reln-to-scattering-sfn} B_{2\gamma-2j}^{2\gamma}(V) & = (-1)^{\lfloor\gamma\rfloor-j+1}2^{2\lfloor\gamma\rfloor-2j+1}(\lfloor\gamma\rfloor-j)!\frac{\Gamma(1-j+\gamma)}{\Gamma([\gamma])}\hf,
 \end{align}
 where $\hf:=S\bigl(\frac{n}{2}+\cgamma\bigr)f$.  Moreover, $B_{2\alpha}^{2\gamma}(V)=0$ for all $B_{2\alpha}^{2\gamma}\in\mB^{2\gamma}\setminus\{B_{2j}^{2\gamma},B_{2\gamma-2j}^{2\gamma}\}$.
\end{prop}

\begin{proof}
 To begin, note that
 \begin{equation}
  \label{eqn:basic-rules}
  \begin{split}
   T\left(y^{2j}\right) & = 4j(j-[\gamma])y^{2j-2}, \\
   T\left(y^{2[\gamma]+2j}\right) & = 4j(j+[\gamma])y^{2[\gamma]+2j-2}
  \end{split}
 \end{equation}
 for all $j\in\bN_0$.
 
 By~\eqref{eqn:poisson_asymptotics}, it holds that
 \[ \mP\left(\frac{n}{2}+\cgamma\right)f = U + \widehat{U} , \]
 where $U$ solves $y^{-s}U\to f$ as $y\to0$ and $\Delta_{g_+}U+s(n-s)U=O(y^\infty)$ with $s=\frac{n}{2}-\cgamma$ and $\widehat{U}$ solves $y^{-\widehat{s}}\widehat{U}\to\hf$ as $y\to0$ and $\Delta_{g_+}\widehat{U}+\widehat{s}(n-\widehat{s})\widehat{U}=O(y^\infty)$ with $\widehat{s}=\frac{n}{2}+\cgamma$.  It follows from Proposition~\ref{prop:formal_solutions} that
 \begin{multline}
  \label{eqn:V-expansion}
  V = \sum_{\ell=0}^\infty \frac{\Gamma(1+2j-\gamma)}{2^{2\ell}\ell!\Gamma(\ell+1+2j-\gamma)}y^{2j+2\ell} (-\oDelta)^\ell f \\
   + \sum_{\ell=0}^\infty \frac{\Gamma(1-2j+\gamma)}{2^{2\ell}\ell!\Gamma(\ell+1-2j+\gamma)}y^{2\gamma-2j+2\ell} (-\oDelta)^\ell\hf . 
 \end{multline}
 We separate the proof into two cases:
 
 First consider $B_{2\ell}^{2\gamma}$ for integers $0\leq\ell\leq\lfloor\gamma\rfloor$.  From~\eqref{eqn:basic-rules} and~\eqref{eqn:V-expansion} we immediately deduce that $B_{2\ell}^{2\gamma}(V)=0$ if $\ell<j$ and that~\eqref{eqn:reln-to-scattering-fn} holds.  Suppose now that there is an integer $\ell_0\geq0$ such that $B_{2j+2\ell}^{2\gamma}(V)=0$ for all integers $1\leq\ell\leq\ell_0$; note that this holds trivially when $\ell_0=0$.  We compute that
 \begin{multline*}
  B_{2j+2\ell+2}^{2\gamma}(V) = (-1)^{j}2^{2j}\frac{(j+\ell+1)!\Gamma(2+j+\ell-[\gamma])\Gamma(1+2j-\gamma)}{(\ell+1)!\Gamma(1-[\gamma])\Gamma(2+2j+\ell-\gamma)}\oDelta^{\ell+1}f \\ - \binom{j+\ell+1}{\ell+1}\frac{\Gamma(2+j+\ell-[\gamma])\Gamma(1+2j-\gamma)}{\Gamma(1+j-[\gamma])\Gamma(2+2j+\ell-\gamma)}\oDelta^{\ell+1}B_{2j}^{2\gamma}(V) = 0 .
 \end{multline*}
 The claim follows by induction.
 
 Next consider $B_{2[\gamma]+2\ell}^{2\gamma}$ for integers $0\leq\ell\leq\lfloor\gamma\rfloor$.  Computing as above, we deduce that~\eqref{eqn:reln-to-scattering-sfn} holds and that $B_{2[\gamma]+2\ell}^{2\gamma}(V)=0$ if $\ell\not=\lfloor\gamma\rfloor-j$.
\end{proof}

\begin{prop}
 \label{prop:reln-to-scattering2}
 Let $\gamma\in(0,\infty)\setminus\bN$.  Given $0\leq j\leq\lfloor\gamma\rfloor-\lfloor\gamma/2\rfloor-1$, set $\cgamma:=\lfloor\gamma\rfloor-[\gamma]-2j$.  Let $V=y^{-\frac{n-2\gamma}{2}}\mP\bigl(\frac{n}{2}+\cgamma\bigr)\phi$ for some $\phi\in C^\infty(\bR^n)\cap H^{\cgamma}(\bR^n)$.  It holds that
 \begin{align}
  \label{eqn:reln-to-scattering2-fn} B_{2[\gamma]+2j}^{2\gamma}(V) & = (-1)^{j+1}2^{2j+1}j!\frac{\Gamma(1+j+[\gamma])}{\Gamma([\gamma])}\phi, \\
  \label{eqn:reln-to-scattering2-sfn} B_{2\lfloor\gamma\rfloor-2j}^{2\gamma}(V) & = (-1)^{\lfloor\gamma\rfloor-j}2^{2\lfloor\gamma\rfloor-2j}(\lfloor\gamma\rfloor-j)!\frac{\Gamma(1+\lfloor\gamma\rfloor-j-[\gamma])}{\Gamma(1-[\gamma])}\hphi, 
 \end{align}
 where $\hphi:=S\bigl(\frac{n}{2}+\cgamma\bigr)\phi$.  Moreover, $B_{2\alpha}^{2\gamma}(V)=0$ for all $B_{2\alpha}^{2\gamma}\in\mB^{2\gamma}\setminus\{ B_{2[\gamma]+2j}^{2\gamma},B_{2\lfloor\gamma\rfloor-2j}^{2\gamma}\}$.
\end{prop}

\begin{proof}
 It follows from Proposition~\ref{prop:formal_solutions} that
 \begin{multline}
  \label{eqn:V-expansion2}
  V = \sum_{\ell=0}^\infty \frac{\Gamma(1-\lfloor\gamma\rfloor+2j+[\gamma])}{2^{2\ell}\ell!\Gamma(1+\ell-\lfloor\gamma\rfloor+2j+[\gamma])}y^{2[\gamma]+2j+2\ell}(-\oDelta)^\ell\phi \\
   + \sum_{\ell=0}^\infty \frac{\Gamma(1+\lfloor\gamma\rfloor-2j-[\gamma])}{2^{2\ell}\ell!\Gamma(1+\ell+\lfloor\gamma\rfloor-2j-[\gamma])}y^{2\lfloor\gamma\rfloor-2j+2\ell}(-\oDelta)^\ell\hphi. 
 \end{multline}
 Using~\eqref{eqn:basic-rules} and~\eqref{eqn:V-expansion2} and computing as in the proof of Proposition~\ref{prop:reln-to-scattering} yields the desired result.
\end{proof}

Let
\begin{multline}
 \label{eqn:dirichlet-form}
 \mQ_{2\gamma}(U,V) := \int_{\bR_+^{n+1}} U\,L_{2k}V\,y^m\,dx\,dy + \sum_{j=0}^{\lfloor\gamma/2\rfloor} \oint_{\bR^n} B_{2j}^{2\gamma}(U)\,B_{2\gamma-2j}^{2\gamma}(V)\,dx \\ - \sum_{j=0}^{\lfloor\gamma\rfloor-\lfloor\gamma/2\rfloor-1} \oint_{\bR^n} B_{2[\gamma]+2j}^{2\gamma}(U)\,B_{2\lfloor\gamma\rfloor-2j}^{2\gamma}(V)\,dx
\end{multline}
be the Dirichlet form determined by $\gamma\in(0,\infty)\setminus\bN$.  The second goal of this section is to prove that $\mQ_{2\gamma}$ is symmetric.  This implies that boundary value problems involving $L_{2k}$ and $\mB^{2\gamma}$ are variational (e.g.\ Theorem~\ref{thm:dirichlet-principle}).  It also implies that the generalized Dirichlet-to-Neumann operators associated to $L_{2k}$ are formally self-adjoint (cf.\ \cite{Case2015fi,Case2015b,CaseLuo2018}).

The proof that $\mQ_{2\gamma}$ is symmetric is essentially a lengthy computation involving integration by parts.  To that end, it is useful to express $\iota^\ast\circ\Delta_m^j$ and $\iota^\ast\circ y^m\partial_y\Delta_m^j$ in terms of the boundary operators of Definition~\ref{defn:operators}.

\begin{prop}
 \label{prop:operators-via-laplacian}
 Let $\gamma\in(0,\infty)\setminus\bN$.  Given $0\leq j\leq\lfloor\gamma\rfloor$, it holds that
 \begin{align*}
  \iota^\ast\circ \Delta_m^j & = \sum_{\ell=0}^j (-1)^\ell\binom{j}{\ell}\frac{(\lfloor\gamma\rfloor-\ell)!}{(\lfloor\gamma\rfloor-j)!} \frac{\Gamma(\gamma-j-\ell)}{\Gamma(\gamma-2\ell)} \oDelta^{j-\ell} B_{2\ell}^{2\gamma}, \\
  \iota^\ast\circ y^m\partial_y\Delta_m^j & = (-1)^{j+1}\sum_{\ell=0}^j \binom{j}{\ell}\frac{(\lfloor\gamma\rfloor-\ell)!}{(\lfloor\gamma\rfloor-j)!}\frac{\Gamma(1+2\ell-\lfloor\gamma\rfloor+[\gamma])}{\Gamma(1+j+\ell-\lfloor\gamma\rfloor+[\gamma])}\oDelta^{j-\ell} B_{2[\gamma]+2\ell}^{2\gamma} .
 \end{align*}
\end{prop}

\begin{proof}
 First we compute $\iota^\ast\circ \Delta_m^j$.  Since $\Delta_m=T+\oDelta$, we readily deduce from Definition~\ref{defn:operators} that
 \begin{equation}
  \label{eqn:Deltam-preidentity}
  \iota^\ast\circ \Delta_m^j = \sum_{\ell=0}^j (-1)^\ell\binom{j}{\ell}\frac{\Gamma(1+2\ell-\gamma)}{\Gamma(1+\ell-[\gamma])}F(\ell,j-\ell,\gamma-j)\oDelta^{j-\ell} B_{2\ell}^{2\gamma} ,
 \end{equation}
 where
 \begin{equation}
  \label{eqn:F-defn}
  F(j,\ell,\gamma) := \sum_{s=0}^\ell (-1)^s\binom{\ell}{s}\frac{\Gamma(1+j+s-[\gamma])}{\Gamma(1+j-\ell+s-\gamma)} .
 \end{equation}
 A straightforward induction argument yields
 \begin{equation}
  \label{eqn:F-eval}
  F(j,\ell,\gamma) = (-1)^\ell \frac{(\lfloor\gamma\rfloor+\ell)!}{\lfloor\gamma\rfloor!}\frac{\Gamma(1+j-[\gamma])}{\Gamma(1+j-\gamma)} .
 \end{equation}
 Inserting~\eqref{eqn:F-eval} into~\eqref{eqn:Deltam-preidentity} and using the identity $\Gamma(z)\Gamma(1-z)=\pi/\sin(\pi z)$ yields the desired formula for $\iota^\ast\circ \Delta_m^j$.
 
 Similarly, the identity $\Delta_m=T+\oDelta$ and Definition~\ref{defn:operators} together yield
 \begin{multline*}
  \iota^\ast\circ y^m\partial_y\Delta_m^j = \sum_{\ell=0}^j (-1)^{\ell+1}\binom{j}{\ell}\frac{\Gamma(1+2\ell-\lfloor\gamma\rfloor+[\gamma])}{\Gamma(1+\ell+[\gamma])} \\ \times F(\ell+1,j-\ell,1-[\gamma]+\lfloor\gamma\rfloor-j)\oDelta^{j-\ell}B_{2[\gamma]+2\ell}^{2j} ,
 \end{multline*}
 where $F$ is given by~\eqref{eqn:F-defn}.  Combining the above display with~\eqref{eqn:F-eval} yields the desired formula for $\iota^\ast\circ y^m\partial_y\Delta_m^j$.
\end{proof}

\begin{remark}
 The result of Proposition~\ref{prop:operators-via-laplacian} gives recursive formulas for $B_{2j}^{2\gamma}$ (resp.\ $B_{2[\gamma]+2j}^{2\gamma}$) in terms of $\iota^\ast\circ \Delta_m^j$ (resp.\ $\iota^\ast\circ y^m\partial_y\Delta_m^j$) and $\oDelta$.  It is possible to solve these recursive relations to deduce formulas for the boundary operators involving only interior and tangential Laplacians and the weighted normal derivative.  These operators will necessarily be, up to a choice of sign, the highest-order terms of the boundary operators associated to weighted GJMS operators (cf.\ \cite{Case2015fi,Case2015b,CaseLuo2018}). 
\end{remark}

We now prove that $\mQ_{2\gamma}$ is symmetric by giving an explicit formula for $\mQ_{2\gamma}$.  Especially notable in this formula is that the boundary integration involves only the Dirichlet data
\[ \mB_D^{2\gamma} := \left\{ B_{2\alpha}^{2\gamma}\in\mB^{2\gamma} \suchthat 0 \leq \alpha < \gamma/2 \right\} \]
of the inputs.

\begin{thm}
 \label{thm:symmetry}
 Let $\gamma\in(0,\infty)\setminus\bN$.  Set $k=\lfloor\gamma\rfloor+1$ and $m=1-2[\gamma]$.  Then
 \begin{multline*}
  \mQ_{2\gamma}(U,V) = \int_{\bR_+^{n+1}} \lp \Delta_m^{k/2}U, \Delta_m^{k/2}V\rp\,y^m\,dx\,dy \\
   - \sum_{j=0}^{\lfloor\gamma/2\rfloor}\sum_{\ell=0}^{\lfloor\gamma\rfloor-\lfloor\gamma/2\rfloor-1}\oint_{\bR^n} C(j,\ell,\gamma) \Bigl(B_{2j}^{2\gamma}(U)\oDelta^{\lfloor\gamma\rfloor-j-\ell} B_{2[\gamma]+2\ell}^{2\gamma}(V) \\ + B_{2j}^{2\gamma}(V)\oDelta^{\lfloor\gamma\rfloor-j-\ell} B_{2[\gamma]+2\ell}^{2\gamma}(U)\Bigr)\,dx
 \end{multline*}
 where
 \[ \lp\Delta_m^{k/2}U,\Delta_m^{k/2}V\rp :=
      \begin{cases}
       (\Delta_m^{k/2}U)(\Delta_m^{k/2}V), & \text{if $k$ is even}, \\
       \lp\nabla\Delta_m^{(k-1)/2}U,\nabla\Delta_m^{(k-1)/2}V\rp, & \text{if $k$ is odd}, 
      \end{cases} \]
 and
 \begin{multline*}
  C(j,\ell,\gamma) := (-1)^{\lfloor\gamma/2\rfloor+j}\frac{(\lfloor\gamma\rfloor-\ell)!}{j!}\binom{\lfloor\gamma\rfloor-j}{\ell}\binom{\lfloor\gamma\rfloor-j-\ell-1}{\lfloor\gamma/2\rfloor-j} \\ \times \frac{\Gamma(1-\lfloor\gamma\rfloor+2\ell+[\gamma])\Gamma(\gamma-\lfloor\gamma/2\rfloor-j)}{([\gamma]-j+\ell)\Gamma(\gamma-2j)\Gamma([\gamma]-\lfloor\gamma/2\rfloor+\ell)} .
 \end{multline*}
 In particular, $\mQ_{2\gamma}$ is symmetric.
\end{thm}

\begin{proof}
 To begin, set
 \[ \mQ_{2\gamma}^{(0)}(U,V) := \int_{\bR_+^{n+1}} \left(U\,L_{2k}(V) - \lp\Delta_m^{k/2}U,\Delta_m^{k/2}V\rp\right)\,y^m\,dx\,dy . \]
 A direct computation shows that
 \begin{equation}
  \label{eqn:ibp}
  \mQ_{2\gamma}^{(0)}(U,V) = \mB_1(U,V) - \mB_2(U,V),
 \end{equation}
 where
 \begin{align*}
  \mB_1(U,V) & = (-1)^{\lfloor\gamma\rfloor}\sum_{j=0}^{\lfloor\gamma/2\rfloor} \oint_{\bR^n} (\Delta_m^jU)y^m\partial_y\Delta_m^{\lfloor\gamma\rfloor-j}V\,dx , \\
  \mB_2(U,V) & = (-1)^{\lfloor\gamma\rfloor}\sum_{j=0}^{\lfloor\gamma\rfloor-\lfloor\gamma/2\rfloor-1}\oint_{\bR^n} (\Delta_m^{\lfloor\gamma\rfloor-j}V)y^m\partial_y\Delta_m^jU\,dx .
 \end{align*}

 We begin by simplifying $\mB_1$.  It follows from Proposition~\ref{prop:operators-via-laplacian} that
 \begin{multline*}
  \mB_1(U,V) = \sum_{j=0}^{\lfloor\gamma/2\rfloor} \sum_{\ell=0}^j \sum_{s=0}^{\lfloor\gamma\rfloor-j} \oint_{\bR^n} (-1)^{j+\ell+1}\binom{j}{\ell}\binom{\lfloor\gamma\rfloor-j}{s} \frac{(\lfloor\gamma\rfloor-\ell)!(\lfloor\gamma\rfloor-s)!}{(\lfloor\gamma\rfloor-j)!j!} \\ \times\frac{\Gamma(\gamma-j-\ell)\Gamma(1+2s-\lfloor\gamma\rfloor+[\gamma])}{\Gamma(\gamma-2\ell)\Gamma(1+s-j+[\gamma])}B_{2\ell}^{2\gamma}(U)\oDelta^{\lfloor\gamma\rfloor-\ell-s}B_{2[\gamma]+2s}^{2\gamma}(V)\,dx
 \end{multline*}
 Reindexing the summations in terms of $\ell$, $s$, and $j-\ell$ yields
 \begin{align*}
  \mB_1(U,V) & = -\sum_{\ell=0}^{\lfloor\gamma/2\rfloor}\sum_{s=\lfloor\gamma\rfloor-\lfloor\gamma/2\rfloor}^{\lfloor\gamma\rfloor-\ell} \oint_{\bR^n} \binom{\lfloor\gamma\rfloor-\ell}{s}\frac{(\lfloor\gamma\rfloor-s)!}{\ell!}\frac{\Gamma(1+2s-\lfloor\gamma\rfloor+[\gamma])}{\Gamma(\gamma-2\ell)} \\
    & \qquad \times H(\lfloor\gamma\rfloor-\ell-s,\lfloor\gamma\rfloor-\ell-s,\gamma-2\ell)B_{2\ell}^{2\gamma}(U)\oDelta^{\lfloor\gamma\rfloor-\ell-s}B_{2[\gamma]+2s}^{2\gamma}(V)\,dx \\
   & \quad - \sum_{\ell=0}^{\lfloor\gamma/2\rfloor}\sum_{s=0}^{\lfloor\gamma\rfloor-\lfloor\gamma/2\rfloor-1} \oint_{\bR^n} \binom{\lfloor\gamma\rfloor-\ell}{s}\frac{(\lfloor\gamma\rfloor-s)!}{\ell!} \frac{\Gamma(1+2s-\lfloor\gamma\rfloor+[\gamma])}{\Gamma(\gamma-2\ell)} \\
    & \qquad \times H\left(\lfloor\gamma/2\rfloor-\ell,\lfloor\gamma\rfloor-\ell-s,\gamma-2\ell\right)B_{2\ell}^{2\gamma}(U)\oDelta^{\lfloor\gamma\rfloor-\ell-s}B_{2[\gamma]+2s}^{2\gamma}(V)\,dx
 \end{align*}
 where
 \[ H(n,d,\gamma) := \sum_{\ell=0}^n (-1)^\ell\binom{d}{\ell}\frac{\Gamma(\gamma-\ell)}{\Gamma(1+\gamma-d-\ell)} \]
 for $n,d\in\bN_0$ and $\gamma\in\bR$.  A straightforward induction argument shows that
 \begin{equation}
  \label{eqn:H-eval}
  H(n,d,\gamma) = (-1)^{n}\binom{d-1}{n}\frac{\Gamma(\gamma-n)}{(\gamma-d)\Gamma(\gamma-n-d)} ,
 \end{equation}
 with the convention that $\binom{-1}{0}=1$.  Therefore
 \begin{multline*}
  \mB_1(U,V) = -\sum_{j=0}^{\lfloor\gamma/2\rfloor}\oint_{\bR^n} B_{2j}^{2\gamma}(U)B_{2\gamma-2j}^{2\gamma}(V)\,dx \\ - \sum_{j=0}^{\lfloor\gamma/2\rfloor}\sum_{\ell=0}^{\lfloor\gamma\rfloor-\lfloor\gamma/2\rfloor-1} \oint_{\bR^n} (-1)^{\lfloor\gamma/2\rfloor+j}\frac{(\lfloor\gamma\rfloor-\ell)!}{j!}\binom{\lfloor\gamma\rfloor-j}{\ell}\binom{\lfloor\gamma\rfloor-j-\ell-1}{\lfloor\gamma/2\rfloor-j} \\ \times \frac{\Gamma(1+2\ell-\lfloor\gamma\rfloor+[\gamma])\Gamma(\gamma-\lfloor\gamma/2\rfloor-j)}{([\gamma]-j+\ell)\Gamma(\gamma-2j)\Gamma([\gamma]-\lfloor\gamma/2\rfloor+\ell)} B_{2j}^{2\gamma}(U)\oDelta^{\lfloor\gamma\rfloor-j-\ell}B_{2[\gamma]+2\ell}^{2\gamma}(V)\,dx
 \end{multline*}
 
 We now simplify $\mB_2$.  Following the strategy used to simplify $\mB_1$, we deduce from Proposition~\ref{prop:operators-via-laplacian} and reindexing that
 \begin{align*}
  \mB_2(U,V) & = \sum_{\ell=0}^{\lfloor\gamma\rfloor-\lfloor\gamma/2\rfloor-1}\sum_{s=\lfloor\gamma/2\rfloor+1}^{\lfloor\gamma\rfloor-\ell} \oint_{\bR^n} \binom{\lfloor\gamma\rfloor-\ell}{s}\frac{(\lfloor\gamma\rfloor-s)!}{\ell!}\frac{\Gamma(1+2\ell-\lfloor\gamma\rfloor+[\gamma])}{\Gamma(\gamma-2s)} \\
    & \qquad \times H\left(\lfloor\gamma\rfloor-\ell-s,\lfloor\gamma\rfloor-\ell-s,\lfloor\gamma\rfloor-[\gamma]-2\ell\right)B_{2s}^{2\gamma}(V)\oDelta^{\lfloor\gamma\rfloor-\ell-s}B_{2[\gamma]+2\ell}^{2\gamma}(U)\,dx \\
   & \quad + \sum_{\ell=0}^{\lfloor\gamma\rfloor-\lfloor\gamma/2\rfloor-1}\sum_{s=0}^{\lfloor\gamma/2\rfloor} \oint_{\bR^n} \binom{\lfloor\gamma\rfloor-\ell}{s}\frac{(\lfloor\gamma\rfloor-s)!}{\ell!}\frac{\Gamma(1+2\ell-\lfloor\gamma\rfloor+[\gamma])}{\Gamma(\gamma-2s)} \\
    & \qquad \times H\left(\lfloor\gamma\rfloor-\lfloor\gamma/2\rfloor-\ell-1,\lfloor\gamma\rfloor-\ell-s,\lfloor\gamma\rfloor-[\gamma]-2\ell\right)B_{2s}^{2\gamma}(V)\oDelta^{\lfloor\gamma\rfloor-\ell-s}B_{2[\gamma]+2\ell}^{2\gamma}(U)\,dx .
 \end{align*}
 Applying~\eqref{eqn:H-eval}, we conclude that
 \begin{align*}
  \mB_2(U,V) & = -\sum_{j=0}^{\lfloor\gamma\rfloor-\lfloor\gamma/2\rfloor-1} \oint_{\bR^n} B_{2[\gamma]+2j}^{\gamma}(U) B_{2\lfloor\gamma\rfloor-2j}^{2\gamma}(V)\,dx \\
    & \quad + \sum_{j=0}^{\lfloor\gamma\rfloor-\lfloor\gamma/2\rfloor-1} \sum_{\ell=0}^{\lfloor\gamma/2\rfloor} \oint_{\bR^n} (-1)^{\lfloor\gamma/2\rfloor+\ell}\frac{(\lfloor\gamma\rfloor-\ell)!}{j!}\binom{\lfloor\gamma\rfloor-j}{\ell}\binom{\lfloor\gamma\rfloor-j-\ell-1}{\lfloor\gamma/2\rfloor-\ell} \\
     & \qquad \times \frac{\Gamma(1+2j-\lfloor\gamma\rfloor+[\gamma])\Gamma(\gamma-\lfloor\gamma/2\rfloor-\ell)}{([\gamma]-\ell+j)\Gamma(\gamma-2\ell)\Gamma([\gamma]-\lfloor\gamma/2\rfloor+j)}B_{2\ell}^{2\gamma}(V)\oDelta^{\lfloor\gamma\rfloor-j-\ell}B_{2[\gamma]+2j}^{2\gamma}(U)\,dx 
 \end{align*}
 
 Inserting the final expressions for $\mB_1$ and $\mB_2$ into~\eqref{eqn:ibp} and using the definition of $\mQ_{2\gamma}$ yields the desired conclusion.
\end{proof}
\section{Conformal covariance}
\label{sec:conformal}

In this section we establish the conformal covariance of the weighted poly-Laplacian $L_{2k}$ and the boundary operators $\mB^{2\gamma}$ determined by a given $\gamma\in(0,\infty)\setminus\bN$.  The former result is readily deduced from~\cite{Case2011t,CaseChang2013}, while the latter result is new.

To prove these results we need to commute powers of $r^2(x,y):=\lv x\rv^2+y^2$, regarded as a multiplication operator, through powers of the weighted Laplacian $\Delta_m$ and their composition with $\iota^\ast y^m\partial_y$.  This is summarized in the following two lemmas.

\begin{lem}
 \label{lem:Deltam_rs}
 Let $\gamma\in(0,\infty)\setminus\bN$, set $m:=1-2[\gamma]$, and define $\Delta_m$ by~\eqref{eqn:weighted_laplacian}.  Denote $r^2(x,y):=\lv x\rv^2+y^2$.  Let $k\in\bN_0$ and $s\in\bR$, and regard $r^{2s}$ as a multiplication operator.  As operators, it holds that
 \begin{multline*}
  \Delta_m^kr^{2s} = \sum_{j=0}^k\sum_{\ell=0}^j 2^{2j-\ell}\binom{k}{j}\binom{j}{\ell} \\ \times \frac{\Gamma(s+1)\Gamma\bigl(\frac{n}{2}+1+s+k-j-[\gamma]\bigr)}{\Gamma(s+1-j)\Gamma\bigl(\frac{n}{2}+1+s+k-2j+\ell-[\gamma]\bigr)}r^{2s-2j}\nabla_{\nabla r^2}^\ell\Delta_m^{k-j},
 \end{multline*}
 where
 \[ \nabla_{\nabla r^2}^\ell U := \nabla^\ell U(\nabla r^2,\dotsc,\nabla r^2) \]
 for all $U\in\mC^{2\gamma}$.
\end{lem}

\begin{remark}
 \label{rk:Deltam_rs}
 We will also apply Lemma~\ref{lem:Deltam_rs} in $\bR^n$ with the induced Laplacian, where it holds that
 \begin{multline*}
  \oDelta^k\orr^{2s} = \sum_{j=0}^k\sum_{\ell=0}^j 2^{2j-\ell}\binom{k}{j}\binom{j}{\ell} \\ \times \frac{\Gamma(s+1)\Gamma\bigl(\frac{n}{2}+s+k-j\bigr)}{\Gamma(s+1-j)\Gamma\bigl(\frac{n}{2}+s+k-2j+\ell\bigr)}\orr^{2s-2j}\onabla_{\onabla \orr^2}^\ell\oDelta^{k-j},
 \end{multline*}
 where $\orr^2(x,y):=\lv x\rv^2$ and
 \[ \onabla_{\onabla\orr^2}^\ell u := \onabla^\ell u(\onabla\orr^2,\dotsc,\onabla\orr^2) \]
 for all $u\in C^\infty(\bR^n)$.
\end{remark}

\begin{proof}
 Observe that
 \[ \Delta_m(r^{2s}u) = r^{2s}\Delta_mu + 2sr^{2s-2}\lp\nabla r^{2s},\nabla u\rp + 2s(n+2s-2[\gamma])r^{2s-2}u \]
 for all $u\in C^\infty(\bR_+^{n+1})$.  Hence, as operators,
 \begin{equation}
  \label{eqn:Delta-onestep}
  \Delta_mr^{2s} = r^{2s}\Delta_m + 2s(n+2s-2[\gamma])r^{2s-2} + 2sr^{2s-2}\nabla_{\nabla r^2} .
 \end{equation}
 
 Let $\ell\in\bN_0$.  The fact that $\nabla^2r^2=2(dx^2+dy^2)$ implies that
 \begin{align}
  \label{eqn:DeltaHessfn} \Delta\nabla_{\nabla r^2}^\ell u & = \nabla_{\nabla r^2}^\ell\Delta u + 4\ell\nabla_{\nabla r^2}^{\ell-1}\Delta u + 4\ell(\ell-1)\nabla_{\nabla r^2}^{\ell-2}\Delta u , \\
  \label{eqn:gradHessfn} \lp y^{-1}\nabla y,\nabla_{\nabla r^2}^\ell u\rp & = \nabla^{\ell+1}u(y^{-1}\nabla y,\nabla r^2,\dotsc,\nabla r^2) \\
   \notag & \quad + 2\ell\nabla^\ell u(y^{-1}\nabla y,\nabla r^2,\dotsc,\nabla r^2)
 \end{align}
 for all $u\in C^\infty(\bR_+^{n+1})$.  The fact that $\nabla^2y=0$ implies that
 \[ \nabla^\ell (y^{-1}\nabla y)(X,\nabla r^2,\dotsc,\nabla r^2) = (-2)^\ell \ell! y^{-1}\lp\nabla y,X\rp \]
 for all vector fields $X$ on $\bR_+^{n+1}$.  In particular,
 \[ \nabla_{\nabla r^2}^\ell\lp\nabla u,y^{-1}\nabla r^2\rp = \sum_{j=0}^\ell (-2)^j\frac{\ell!}{(\ell-j)!}\nabla^{\ell+1-j}u(y^{-1}\nabla y,\nabla r^2,\dotsc,\nabla r^2) . \]
 Combining this with~\eqref{eqn:gradHessfn} yields
 \begin{multline*}
  \lp y^{-1}\nabla y,\nabla_{\nabla r^2}^\ell u\rp = \nabla_{\nabla r^2}^\ell\lp\nabla u,y^{-1}\nabla y\rp + 4\ell\nabla_{\nabla r^2}^{\ell-1}\lp\nabla u,y^{-1}\nabla y\rp \\+ 4\ell(\ell-1)\nabla_{\nabla r^2}^{\ell-2}\lp\nabla u,y^{-1}\nabla y\rp .
 \end{multline*}
 Combining this with~\eqref{eqn:DeltaHessfn} implies that, as operators,
 \begin{equation}
  \label{eqn:Delta-twostep}
  \Delta_m\nabla_{\nabla r^2}^\ell  = 4\ell(\ell-1)\nabla_{\nabla r^2}^{\ell-2}\Delta_m + 4\ell\nabla_{\nabla r^2}^{\ell-1}\Delta_m + \nabla_{\nabla r^2}^\ell\Delta_m .
 \end{equation}
 
 Finally, combining~\eqref{eqn:Delta-onestep} and~\eqref{eqn:Delta-twostep} with a simple induction argument yields the desired conclusion.
\end{proof}

\begin{lem}
 \label{lem:ymDeltam_rs}
 Let $\gamma\in(0,\infty)\setminus\bN$, set $m:=1-2[\gamma]$, and define $\Delta_m$ by~\eqref{eqn:weighted_laplacian}.  Denote $r^2(x,y):=\lv x\rv^2+y^2$.  Let $k\in\bN_0$ and $s\in\bR$, and regard $r^{2s}$ as a multiplication operator.  As operators, it holds that
 \begin{multline*}
  \iota^\ast\circ y^m\partial_y\Delta_m^kr^{2s} = \sum_{j=0}^k\sum_{\ell=0}^j 2^{2j-\ell}\binom{k}{j}\binom{j}{\ell} \\ \times \frac{\Gamma(s+1)\Gamma\bigl(\frac{n}{2}+1+[\gamma]+s+k-j\bigr)}{\Gamma(s+1-j)\Gamma\bigl(\frac{n}{2}+1+[\gamma]+s+k-2j+\ell\bigr)}\orr^{2s-2j}\onabla_{\onabla \orr^2}^\ell\circ\iota^\ast\circ y^m\partial_y\Delta_m^{k-j},
 \end{multline*}
 where $\orr$ and $\onabla_{\onabla\orr^2}^\ell$ are as in Remark~\ref{rk:Deltam_rs}.
\end{lem}

\begin{proof}
 On the one hand, the identity
 \[ \nabla_{\nabla r^2}^k\circ\nabla_{\nabla r^2} = \nabla_{\nabla r^2}^{k+1} + 2k\nabla_{\nabla r^2}^k \]
 and a straightforward induction argument imply that
 \[ \iota^\ast\circ y^m\partial_y\circ\nabla_{\nabla r^2}^k = \sum_{j=0}^k 2^j\binom{k}{j}\frac{\Gamma(1+2[\gamma])}{\Gamma(1+2[\gamma]-j)}\onabla_{\onabla\orr^2}^{k-j}\circ\iota^\ast\circ y^m\partial_y \]
 on $\mC^{2\gamma}$.  On the other hand, it is clear that $\iota^\ast\circ y^m\partial_y\circ r^{2s} = \orr^{2s}\iota^\ast\circ y^m\partial_y$ on $\mC^{2\gamma}$.  Combining these observations with Lemma~\ref{lem:Deltam_rs} yields
 \begin{multline*}
  \iota^\ast\circ y^m\partial_y\Delta_m^kr^{2s} = \sum_{j=0}^k\sum_{\ell=0}^j 2^{2j-\ell}\binom{k}{j}\binom{j}{\ell}\frac{\Gamma(s+1)}{\Gamma(s+1-j)} \\ \times K\left(\frac{n}{2}+1-[\gamma]+s+k-2j+\ell,1+2[\gamma],j-\ell\right)\orr^{2s-2j}\onabla_{\onabla \orr^2}^{\ell}\circ\iota^\ast\circ y^m\partial_y\Delta_m^{k-j},
 \end{multline*}
 where
 \[ K(a,b,j) := \sum_{t=0}^j \binom{j}{t}\frac{\Gamma(a+j)\Gamma(b)}{\Gamma(a+t)\Gamma(b-t)} . \]
 The conclusion follows from the readily verified identity
 \[ K(a,b,j) = \frac{\Gamma(a+b+j-1)}{\Gamma(a+b-1)} . \qedhere \]
\end{proof}

We are now able to prove the conformal covariance of the weighted poly-Laplacian and its associated boundary operators.  It is clear that these operators are all invariant under translations and rotations which fix the boundary of $\bR_+^{n+1}$.  Thus it suffices to check that they are conformally covariant with respect to the Kelvin transform:

\begin{thm}
 \label{thm:conformal_invariance}
 Fix $\gamma\in(0,\infty)\setminus\bN$.  Let $L_{2k}$ denote the weighted poly-Laplacian~\eqref{eqn:weighted_polylaplacian} and let $\mB^{2\gamma}$ denote the set~\eqref{eqn:boundary_operators} of boundary operators associated to $L_{2k}$ as given by Definition~\ref{defn:operators}.  Let $\hL_{2k}$ and $\hmB^{2\gamma}$ denote the same operators defined in terms of the inverted metric $\hg:=r^{-4}(dx^2+dy^2)$ and $\hy:=r^{-2}y$.  Then, as operators,
 \begin{align}
  \label{eqn:conformal_invariance_Lk} \hL_{2k} & = r^{n-2[\gamma]+4+2\lfloor\gamma\rfloor} L_{2k} r^{2\gamma-n}, \\
  \label{eqn:conformal_invariance_Bk} \hB_{2\alpha}^{2\gamma} & = \orr^{n-2\gamma+4\alpha} B_{2\alpha}^{2\gamma} r^{2\gamma-n}
 \end{align}
 for all $\mB_{2\alpha}^{2\gamma}\in\hmB^{2\gamma}$.
\end{thm}

\begin{proof}
 The well-known conformal invariance of the Laplacian in $\bR^{n+1}$ implies that
 \[ \hDelta = r^{n+3}\Delta r^{1-n} . \]
 On the other hand, it is straightforward to compute that
 \[ \widehat{y^{-1}\partial_y} = r^4y^{-1}\partial_y - 2r^3\partial_r . \]
 From this we readily deduce that
 \begin{equation}
  \label{eqn:conformal_invariance_weighted_laplacian}
  \hDelta_m = r^{3+m+n}\Delta r^{1-m-n} .
 \end{equation}
 A straightforward induction argument using Lemma~\ref{lem:Deltam_rs} and~\eqref{eqn:conformal_invariance_weighted_laplacian} yields
 \begin{equation}
  \label{eqn:conf_Deltamk}
  \hDelta_m^k = r^{2k+1+m+n}\Delta_m^kr^{2k-m-n-1}
 \end{equation}
 for all $k\in\bN_0$.  In particular, \eqref{eqn:conformal_invariance_Lk} holds.
 
 Consider now the operators $B_{2j}^{2\gamma}$, $0\leq j\leq\lfloor\gamma\rfloor$.  It is clear that $B_0^{2\gamma}$ is conformally covariant in the sense of~\eqref{eqn:conformal_invariance_Bk}.  Suppose that $j\in\bN_0$ is such that $B_{2\ell}^{2\gamma}$ is conformally covariant in the sense of~\eqref{eqn:conformal_invariance_Bk} for all $0\leq\ell\leq j$.  It follows from this assumption, Proposition~\ref{prop:operators-via-laplacian} and~\eqref{eqn:conf_Deltamk} that
 \begin{multline*}
  (-1)^{j+1}\hB_{2j+2}^{2\gamma}r^{n-2\gamma} = \orr^{2j+4+n-2[\gamma]}\iota^\ast\circ\Delta_m^{j+1}r^{2j-2\lfloor\gamma\rfloor} \\ - \sum_{\ell=0}^j (-1)^\ell \binom{j+1}{\ell}\frac{(\lfloor\gamma\rfloor-\ell)!}{(\lfloor\gamma\rfloor-j-1)!}\frac{\Gamma(\gamma-j-\ell-1)}{\Gamma(\gamma-2\ell)} \\ \times \orr^{2j+2+n-2\ell}\oDelta^{j+1-\ell}\orr^{2j+2+2\ell-2\gamma}B_{2\ell}^{2\gamma} . 
 \end{multline*}
 It follows readily from Proposition~\ref{prop:operators-via-laplacian} and Lemma~\ref{lem:Deltam_rs} that $B_{2j+2}^{2\gamma}$ is conformally covariant in the sense of~\eqref{eqn:conformal_invariance_Bk}.  Thus~\eqref{eqn:conformal_invariance_Bk} holds for all $\alpha\in\{0,1,\dotsc,\lfloor\gamma\rfloor\}$.
 
 A similar argument using Proposition~\ref{prop:operators-via-laplacian}, Lemma~\ref{lem:Deltam_rs} and Lemma~\ref{lem:ymDeltam_rs} yields~\eqref{eqn:conformal_invariance_Bk} for $\alpha\in\{[\gamma],[\gamma]+1,\dotsc,\gamma\}$.
\end{proof}

\section{The generalized Caffarelli--Silvestre extension}
\label{sec:cs}

The main result of this section is that solutions of the Dirichlet boundary value problem
\begin{equation}
 \label{eqn:bvp}
 \begin{cases}
  L_{2k}V = 0, & \text{in $\bR_+^{n+1}$}, \\
  B_{2j}^{2\gamma}(V) = f^{(2j)}, & \text{for $0\leq j\leq\lfloor\gamma/2\rfloor$}, \\ 
  B_{2[\gamma]+2j}^{2\gamma}(V) = \phi^{(2j)}, & \text{for $0\leq j\leq\lfloor\gamma\rfloor - \lceil\gamma/2\rceil$} ,
 \end{cases}
\end{equation}
can be used to recover the fractional Laplacian $(-\oDelta)^{\cgamma}$ for any $\cgamma\in\mI$.  To that end, we first characterize the solutions of~\eqref{eqn:bvp}:

\begin{thm}
 \label{thm:bvp}
 Let $\gamma\in(0,\infty)\setminus\bN$.  Given functions $f^{(2j)}\in C^\infty(\bR^n)\cap H^{\gamma-2j}(\bR^n)$, $0\leq j\leq\lfloor\gamma/2\rfloor$, and $\phi^{(2j)}\in C^\infty(\bR^n)\cap H^{\lfloor\gamma\rfloor-[\gamma]-2j}(\bR^n)$, $0\leq j\leq \lfloor\gamma\rfloor-\lfloor\gamma/2\rfloor-1$, there is a unique solution $V$ of~\eqref{eqn:bvp}.  Indeed,
 \begin{multline}
  \label{eqn:bvp-solution}
  V = \sum_{j=0}^{\lfloor\gamma/2\rfloor} (-1)^j2^{-2j}\frac{\Gamma(1-[\gamma])}{j!\Gamma(1+j-[\gamma])} F^{(2j)} \\
  + \sum_{j=0}^{\lfloor\gamma\rfloor-\lfloor\gamma/2\rfloor-1} (-1)^{j+1}2^{-2j-1}\frac{\Gamma([\gamma])}{j!\Gamma(1+j+[\gamma])} \Phi^{(2j)} ,
 \end{multline}
 where
 \begin{align*}
  F^{(2j)} & := y^{-\frac{n-2\gamma}{2}}\mP\left(\frac{n}{2}+\gamma-2j\right)f^{(2j)}, \\
  \Phi^{(2j)} & := y^{-\frac{n-2\gamma}{2}}\mP\left(\frac{n}{2}+\lfloor\gamma\rfloor-[\gamma]-2j\right)\phi^{(2j)} .
 \end{align*}
\end{thm}

\begin{proof}
 It follows from Lemma~\ref{lem:kernel}, Proposition~\ref{prop:reln-to-scattering} and Proposition~\ref{prop:reln-to-scattering2} that $V$ satisfies~\eqref{eqn:bvp}.
 
 Suppose now that $U$ is a solution of~\eqref{eqn:bvp}.  Then $W:=U-V$ solves
 \begin{equation}
  \label{eqn:bvp-zero}
  \begin{cases}
   L_{2k}W = 0, & \text{in $\bR_+^{n+1}$}, \\
   B_{2j}^{2\gamma}(W) = 0, & \text{for $0\leq j\leq\lfloor\gamma/2\rfloor$}, \\ 
   B_{2[\gamma]+2j}^{2\gamma}(W) = 0, & \text{for $0\leq j\leq\lfloor\gamma\rfloor - \lfloor\gamma/2\rfloor-1$} .
  \end{cases}
 \end{equation}
 Thus $\mQ_{2\gamma}(W,W)=0$.  It follows from Theorem~\ref{thm:symmetry} that $\Delta_m^{\ell}W=0$, where $\ell=\lfloor(\gamma+1)/2\rfloor$.  If $\gamma\in(0,1)$, we deduce that $W=0$.  If $\gamma>1$, we deduce that $W$ solves the analogue of~\eqref{eqn:bvp-zero} with $\gamma^\prime=\gamma-\lfloor\gamma/2\rfloor-1$.  Continuing in this way in the latter case, we deduce again that $W=0$.  Therefore $U=V$.
\end{proof}

We now present our general analogue of the Caffarelli--Silvestre extension~\cite{CaffarelliSilvestre2007}.  In fact, the following result implies that the fractional Laplacian $(-\oDelta)^{\cgamma}$ can be determined without fully specifying the Dirichlet data (cf.\ \cite{ChangYang2017,Yang2013}).  For example, one can recover $(-\Delta)^\gamma$ by applying $B_{2\gamma}^{2\gamma}$ to any $U\in\ker L_{2k}$ for which $U(\cdot,0)=f$.

\begin{thm}
 \label{thm:cs}
 Let $\gamma\in(0,\infty)\setminus\bN$ and suppose that $V$ is a solution of~\eqref{eqn:bvp}.
 \begin{subequations}
  \label{eqn:cs}
  \begin{enumerate}
   \item Given $0\leq j\leq\lfloor\gamma/2\rfloor$, it holds that
   \begin{equation}
    \label{eqn:cs1}
    B_{2\gamma-2j}^{2\gamma}(V) = c_{\gamma,j}(-\oDelta)^{\gamma-2j} f^{(2j)} ,
   \end{equation}
   where
   \[ c_{\gamma,j} = (-1)^{1+\lfloor\gamma\rfloor}2^{1-2[\gamma]}\frac{(\lfloor\gamma\rfloor-j)!\Gamma(1-[\gamma])\Gamma(1-j+\gamma)\Gamma(2j-\gamma)}{j!\Gamma([\gamma])\Gamma(1+j-[\gamma])\Gamma(-2j+\gamma)} . \]
   \item Given $0\leq j\leq \lfloor\gamma\rfloor-\lfloor\gamma/2\rfloor-1$, it holds that
   \begin{equation}
    \label{eqn:cs2}
    B_{2\lfloor\gamma\rfloor-2j}^{2\gamma}(V) = -d_{\gamma,j}(-\oDelta)^{\lfloor\gamma\rfloor-[\gamma]-2j}\phi^{(2j)} ,
   \end{equation}
   where
   \[ d_{\gamma,j} = (-1)^{\lfloor\gamma\rfloor}2^{2[\gamma]-1}\frac{(\lfloor\gamma\rfloor-j)!\Gamma([\gamma])\Gamma(1+\lfloor\gamma\rfloor-j-[\gamma])\Gamma(2j-\lfloor\gamma\rfloor+[\gamma])}{j!\Gamma(1-[\gamma])\Gamma(1+j+[\gamma])\Gamma(-2j+\lfloor\gamma\rfloor-[\gamma])} . \]
 \end{enumerate}
 \end{subequations}
\end{thm}

\begin{proof}
 By Theorem~\ref{thm:bvp}, $V$ is given by~\eqref{eqn:bvp-solution}.  We separate the proof into two cases $B_{2\alpha}^{2\gamma}\in\mB^{2\gamma}$ depending on whether $\alpha\in\bN_0$ or $\alpha\not\in\bN_0$.
 
 Let $0\leq\gamma\leq\lfloor\gamma/2\rfloor$.  Using~\eqref{eqn:bvp-solution} and Proposition~\ref{prop:reln-to-scattering}, we see that
 \[ B_{2\gamma-2j}^{2\gamma}(V) = (-1)^{\lfloor\gamma\rfloor+1}2^{2\lfloor\gamma\rfloor-4j+1}\frac{(\lfloor\gamma\rfloor-j)!\Gamma(1-j+\gamma)\Gamma(1-[\gamma])}{j!\Gamma([\gamma])\Gamma(1+j-[\gamma])}\hf^{(2j)} , \]
 where $\hf^{(2j)}:=S\bigl(\frac{n}{2}+\gamma-2j\bigr)f^{(2j)}$.  Applying~\eqref{eqn:fractional_gjms_definition} yields~\eqref{eqn:cs1}.
 
 Let $0\leq j\leq\lfloor\gamma\rfloor-\lfloor\gamma/2\rfloor-1$.  Using~\eqref{eqn:bvp-solution} and Proposition~\ref{prop:reln-to-scattering2}, we see that
 \[ B_{2\lfloor\gamma\rfloor-2j}^{2\gamma} = (-1)^{\lfloor\gamma\rfloor+1}2^{2\lfloor\gamma\rfloor-4j-1}\frac{(\lfloor\gamma\rfloor-j)!\Gamma([\gamma])\Gamma(1+\lfloor\gamma\rfloor-j-[\gamma])}{j!\Gamma(1-[\gamma])\Gamma(1+j+[\gamma])}\hphi^{(2j)} , \]
 where $\hphi^{(2j)}:=S\bigl(\frac{n}{2}+\lfloor\gamma\rfloor-[\gamma]-2j\bigr)\phi^{(2j)}$.  Applying~\eqref{eqn:fractional_gjms_definition} yields~\eqref{eqn:cs2}.
\end{proof}
\section{The sharp Sobolev trace inequalities}
\label{sec:trace}

The purpose of this section is to use the boundary operators of Section~\ref{sec:boundary} to prove sharp Sobolev inequalities which imply the Sobolev trace embeddings of the weighted Sobolev space $W^{k,2}(\bR_+^{n+1},y^m)$.  A key tool in this endeavor is the Dirichlet energy
\[ \mE_{2\gamma}(U):=\mQ_{2\gamma}(U,U), \]
where $\mQ_{2\gamma}$ is given by~\eqref{eqn:dirichlet-form}.

Our first result is a Dirichlet principle for solutions of~\eqref{eqn:bvp}.

\begin{thm}
 \label{thm:dirichlet-principle}
 Let $\gamma\in(0,\infty)\setminus\bN$.  Fix functions $f^{(2j)}\in C^\infty(\bR^n)\cap H^{\gamma-2j}(\bR^n)$, $0\leq j\leq\lfloor\gamma/2\rfloor$, and $\phi^{(2j)}\in C^\infty(\bR^n)\cap H^{\lfloor\gamma\rfloor-[\gamma]-2j}(\bR^n)$, $0\leq j\leq\lfloor\gamma\rfloor-\lfloor\gamma/2\rfloor-1$, and denote
 \begin{multline*}
  \mC_D^{2\gamma} := \bigl\{ U\in\mC^{2\gamma} \bigm| B_{2j}^{2\gamma}(U)=f^{(2j)}, 0\leq j\leq\lfloor\gamma\rfloor, \,\mathrm{and}\, \\
   B_{2[\gamma]+2j}^{2\gamma}(U) = \phi^{(2j)}, 0\leq j\leq\lfloor\gamma\rfloor-\lfloor\gamma/2\rfloor-1 \bigr\} .
 \end{multline*}
 Then it holds that
 \[ \mE_{2\gamma}(U) \geq \mE_{2\gamma}(U_D) \]
 for all $U\in\mC_D^{2\gamma}$, where $U_D\in\mC_D^{2\gamma}$ is the unique solution of~\eqref{eqn:bvp}.
\end{thm}

\begin{proof}
 Fix an element $U_0\in\mC_D^{2\gamma}$ and set
 \begin{multline*}
  \mC_0^{2\gamma} = \bigl\{ U\in\mC^{2\gamma} \bigm| B_{2j}^{2\gamma}(U)=0, 0\leq j\leq\lfloor\gamma\rfloor, \,\mathrm{and}\, \\ B_{2[\gamma]+2j}^{2\gamma}(U) = 0, 0\leq j\leq\lfloor\gamma\rfloor-\lfloor\gamma/2\rfloor-1 \bigr\} .
 \end{multline*}
 Observe that
 \[ \mC_D^{2\gamma} = U_0 + \mC_0^{2\gamma} . \]
 Let $V\in\mC_0^{2\gamma}$.  It follows from Theorem~\ref{thm:symmetry} that
 \[ \mE_{2\gamma}(V) = \int_{\bR_+^{n+1}} \lv\Delta_m^{k/2}V\rv^2\,y^m\,dx\,dy \]
 and that
 \[ \frac{d}{dt}\mE_{2\gamma}(U+tV) = 2\mE_{2\gamma}(V) \geq 0 . \]
 Moreover, equality holds if and only if $V\equiv0$, and hence $\mE_{2\gamma}$ is strictly convex in $\mC_D^{2\gamma}$.  Since the solution $U_D\in\mC_D^{2\gamma}$ of~\eqref{eqn:bvp} is a critical point of $\mE_{2\gamma}\colon\mC_D^{2\gamma}\to\bR$, the result follows.
\end{proof}

The following corollary, obtained by evaluating $\mE_{2\gamma}(V)$ using Theorem~\ref{thm:cs}, gives a sharp Sobolev trace inequality for the embedding
\[ W^{\lfloor\gamma\rfloor+1,2}(\bR_+^{n+1},y^{1-2[\gamma]}) \hookrightarrow \bigoplus_{j=0}^{\lfloor\gamma/2\rfloor} H^{\gamma-2j}(\bR^n) \oplus \bigoplus_{j=0}^{\lfloor\gamma\rfloor-\lfloor\gamma/2\rfloor-1} H^{\lfloor\gamma\rfloor-[\gamma]-2j}(\bR^{n}) . \]

\begin{cor}
 \label{cor:energy-trace}
 Let $\gamma\in(0,\infty)\setminus\bN$.  Then
 \begin{multline*}
  \mE_{2\gamma}(U) \geq \sum_{j=0}^{\lfloor\gamma/2\rfloor} c_{\gamma,j}\oint_{\bR^n} f^{(2j)}(-\Delta)^{\gamma-2j}f^{(2j)}\,dx \\ + \sum_{j=0}^{\lfloor\gamma\rfloor-\lfloor\gamma/2\rfloor-1} d_{\gamma,j}\oint_{\bR^n} \phi^{(2j)}(-\oDelta)^{\lfloor\gamma\rfloor-[\gamma]-2j}\phi^{(2j)}\,dx
 \end{multline*}
 for all $U\in\mC^{2\gamma}\cap W^{\lfloor\gamma\rfloor+1,2}(\bR_+^{n+1},y^{1-2[\gamma]})$, where $f^{(2j)}:=B_{2j}^{2\gamma}(U)$, $0\leq j\leq\lfloor\gamma/2\rfloor$ and $\phi^{(2j)}:=B_{2[\gamma]+2j}^{2\gamma}(U)$, $0\leq j\leq\lfloor\gamma\rfloor-\lfloor\gamma/2\rfloor-1$, and the constants $c_{\gamma,j}$ and $d_{\gamma,j}$ are given in Theorem~\ref{thm:cs}.  Moreover, equality holds if and only if $U$ is the solution of~\eqref{eqn:bvp}.
\end{cor}

\begin{proof}
 Set $\mC_D^{2\gamma}=U+\mC_0^{2\gamma}$.  By Theorem~\ref{thm:dirichlet-principle}, there is a unique minimizer $U_D$ of $\mE_{2\gamma}\colon\mC_D^{2\gamma}\to\bR$.  Since $U_D$ satisfies~\eqref{eqn:bvp}, we deduce from Theorem~\ref{thm:cs} that
 \begin{multline*}
  \mE_{2\gamma}(U_D) = \sum_{j=0}^{\lfloor\gamma/2\rfloor} c_{\gamma,j}\oint_{\bR^n} f^{(2j)}(-\Delta)^{\gamma-2j}f^{(2j)}\,dx \\ + \sum_{j=0}^{\lfloor\gamma\rfloor-\lfloor\gamma/2\rfloor-1} d_{\gamma,j}\oint_{\bR^n} \phi^{(2j)}(-\oDelta)^{\lfloor\gamma\rfloor-[\gamma]-2j}\phi^{(2j)}\,dx .
 \end{multline*}
 The conclusion readily follows.
\end{proof}

Specializing to the case when $B_{2\alpha}^{2\gamma}(U)=0$ for $B_{2\alpha}^{2\gamma}\in \mB_D^{2\gamma}\setminus\{B_0^{2\gamma}\}$ yields an energy inequality relating a weighted GJMS operator in the interior and the fractional Laplacian $(-\oDelta)^\gamma$ in the boundary.  This result makes explicit the sharp constants in~\cite[Corollary~3.5]{Yang2013}.

\begin{cor}
 \label{cor:yang}
 Let $\gamma\in(0,\infty)\setminus\bN$ and denote
 \begin{multline*}
  \mC_{+}^{2\gamma} := \bigl\{ U\colon\bR_+^{n+1}\to\bR \bigm| B_{2j}^{2\gamma}(U)=0, 1\leq j\leq\lfloor\gamma\rfloor, \,\mathrm{and}\, \\
   B_{2[\gamma]+2j}^{2\gamma}(U) = 0, 0\leq j\leq\lfloor\gamma\rfloor-\lfloor\gamma/2\rfloor-1 \bigr\} .
 \end{multline*}
 Then
 \begin{equation}
  \label{eqn:yang-energy}
  \mE_{2\gamma}(U) \geq (-1)^{1+\lfloor\gamma\rfloor}2^{1-2[\gamma]}\lfloor\gamma\rfloor!\frac{\gamma\Gamma(-\gamma)}{\Gamma([\gamma])}\oint_{\bR^n} f(-\oDelta)^\gamma f\,dx
 \end{equation}
 for all $U\in\mC_{+}^{2\gamma}$, where $f:=U(\cdot,0)$.  Moreover, equality holds if and only if $U$ is the unique solution of
 \begin{equation}
  \label{eqn:bvp-yang}
  \begin{cases}
   \Delta_m^k U = 0, & \text{in $\bR_+^{n+1}$}, \\
   U = f, & \text{on $\bR^n$}, \\
   \Delta_m^j U = \frac{\lfloor\gamma\rfloor!}{(\lfloor\gamma\rfloor-j)!}\frac{\Gamma(\gamma-j)}{\Gamma(\gamma)}\oDelta^jf, & \text{on $\bR^n$, for $1\leq j\leq\lfloor\gamma/2\rfloor$}, \\
   y^m\partial_y\Delta_m^j U = 0, & \text{on $\bR^n$, for $0\leq j\leq\lfloor\gamma\rfloor-\lfloor\gamma/2\rfloor-1$},
  \end{cases}
 \end{equation}
 where $\Delta_m:=\Delta+(1-2[\gamma])y^{-1}\partial_y$ and $k=\lfloor\gamma\rfloor+1$, and $(-\oDelta)^\gamma f$ can be recovered from the solution of~\eqref{eqn:bvp-yang} by
 \begin{equation}
  \label{eqn:yang-cs}
  \lim_{y\to0^+} y^{1-2[\gamma]}\partial_y\Delta_m^{\lfloor\gamma\rfloor} U = 2^{1-2[\gamma]}\lfloor\gamma\rfloor!\frac{\gamma\Gamma(-\gamma)}{\Gamma([\gamma])}(-\oDelta)^\gamma f  .
 \end{equation}
\end{cor}

\begin{proof}
 It follows immediately from Corollary~\ref{cor:energy-trace} that~\eqref{eqn:yang-energy} holds for all $U\in\mC_+^{2\gamma}$, with equality if and only if $U$ is the unique solution of~\eqref{eqn:bvp} with $B_0^{2\gamma}(U)=f$ and $B_{2\alpha}^{2\gamma}(U)=0$ for $B_{2\alpha}^{2\gamma}\in\mB^{2\gamma}\setminus\{B_0^{2\gamma}\}$.  Using Proposition~\ref{prop:operators-via-laplacian}, we see that~\eqref{eqn:bvp-yang} is equivalent to~\eqref{eqn:bvp} when the latter is restricted to functions $U\in\mC_+^{2\gamma}$.  Using Theorem~\ref{thm:cs}, we see that $B_{2\gamma-2j}^{2\gamma}(U)=0$ for all $1\leq j\leq\lfloor\gamma\rfloor$.  Proposition~\ref{prop:operators-via-laplacian} and Theorem~\ref{thm:cs} then imply the relation~\eqref{eqn:yang-cs}.
\end{proof}

Our next result, obtained by applying the sharp fractional Sobolev inequalities~\cite{Beckner1993,Branson1995,Lieb1983} to Corollary~\ref{cor:energy-trace}, gives a sharp Sobolev trace inequality for the embedding
\[ W^{\lfloor\gamma\rfloor+1,2}(\bR_+^{n+1},y^{1-2[\gamma]}) \hookrightarrow \bigoplus_{j=0}^{\lfloor\gamma/2\rfloor} L^{\frac{2n}{n-2\gamma+4j}}(\bR^n) \oplus \bigoplus_{j=0}^{\lfloor\gamma\rfloor-\lfloor\gamma/2\rfloor-1} L^{\frac{2n}{n-2\lfloor\gamma\rfloor+2[\gamma]+4j}}(\bR^n) \]
when $n>2\gamma$.

\begin{thm}
 \label{thm:sobolev-trace}
 Let $\gamma\in(0,n/2)\setminus\bN$.  Then
 \begin{multline*}
  \mE_{2\gamma}(U) \geq \sum_{j=0}^{\lfloor\gamma/2\rfloor} \frac{\Gamma\bigl(\frac{n}{2}+\gamma-2j\bigr)}{\Gamma\bigl(\frac{n}{2}-\gamma+2j\bigr)}c_{n,\gamma}\Vol(S^n)^{\frac{2\gamma-4j}{n}}\lV f^{(2j)}\rV_{\frac{2n}{n-2\gamma+4j}}^2 \\ + \sum_{j=0}^{\lfloor\gamma\rfloor-\lfloor\gamma/2\rfloor-1} \frac{\Gamma\bigl(\frac{n}{2}+\lfloor\gamma\rfloor-[\gamma]-2j\bigr)}{\Gamma\bigl(\frac{n}{2}-\lfloor\gamma\rfloor+[\gamma]+2j\bigr)}d_{n,\gamma}\Vol(S^n)^{\frac{2\lfloor\gamma\rfloor-2[\gamma]-4j}{n}} \lV\phi^{(2j)}\rV_{\frac{2n}{n-2\lfloor\gamma\rfloor+2[\gamma]+4j}}^2
 \end{multline*}
 for all $U\in\mC^{2\gamma}\cap W^{\lfloor\gamma\rfloor+1,2}(\bR_+^{n+1},y^{1-2[\gamma]})$, where $f^{(2j)}:=B_{2j}^{2\gamma}(U)$, $0\leq j\leq\lfloor\gamma/2\rfloor$, and $\phi^{(2j)}:=B_{2[\gamma]+2j}^{2\gamma}(U)$, $0\leq j\leq\lfloor\gamma\rfloor-\lfloor\gamma/2\rfloor-1$, and the constants $c_{n,\gamma}$ and $d_{n,\gamma}$ are as in Theorem~\ref{thm:cs}.  Moreover, equality holds if and only if $U$ is the unique solution of~\eqref{eqn:bvp} and there are constants $a_j,b_\ell\in\bR$ and $\varepsilon_j,\epsilon_\ell\in(0,\infty)$ and points $\xi_j,\zeta_\ell\in\bR^n$ such that
 \begin{align*}
  f^{(2j)}(x) & = a_j\left(\varepsilon_j+\lv x-\xi_j\rv^2\right)^{-\frac{n-2\gamma+4j}{2}}, && 0 \leq j \leq \lfloor\gamma/2\rfloor, \\
  \phi^{(2\ell)}(x) & = b_j\left(\epsilon_\ell+\lv x-\zeta_\ell\rv^2\right)^{-\frac{n-2\lfloor\gamma\rfloor+2[\gamma]+4j}{2}}, && 0\leq \ell \leq \lfloor\gamma\rfloor - \lfloor\gamma/2\rfloor - 1
 \end{align*}
 for all $x\in\bR^n$.
\end{thm}

\begin{proof}
 Recall (see~\cite{Beckner1993,Branson1995,Lieb1983}) that if $n>2\gamma$, then
 \[ \oint_{\bR^n} f\,(-\oDelta)^\gamma f\,dx \geq \frac{\Gamma\bigl(\frac{n+2\gamma}{2}\bigr)}{\Gamma\bigl(\frac{n-2\gamma}{2}\bigr)}\Vol(S^n)^{\frac{2\gamma}{n}} \lV f\rV_{L^{\frac{2n}{n-2\gamma}}}^2 \]
 for all $f\in H^\gamma(\bR^n)$, with equality if and only if there is are constants $a\in\bR$ and $\varepsilon\in(0,\infty)$ and a point $\xi\in\bR^n$ such that
 \[ f(x) = a\left(\varepsilon+\lv x-\xi\rv^2\right)^{-\frac{n-2\gamma}{2}} . \]
 Combining this with Corollary~\ref{cor:energy-trace} yields the desired conclusion.
\end{proof}

By applying the sharp Onofri inequality~\cite{Beckner1993}, we can also use Corollary~\ref{cor:energy-trace} to prove Theorem~\ref{thm:intro-lebedev-milin}.  This result gives a sharp Lebedev--Milin inequality for the embedding
\[ \mC_+^{2k+1} \cap W^{k+1,2}(\bR_+^{2k+2}) \hookrightarrow e^{L}(\bR^{2k+1}) \]
for any $k\in\bN_0$.  The general sharp inequality for $W^{k+1,2}(\bR_+^{2k+2})$ involves adding extra $L^p$-norms corresponding to $f^{(2j)}$, $1\leq j\leq\lfloor k/2\rfloor$, and $\phi^{(2j)}$, $0\leq j\leq k - \lfloor k/2\rfloor-1$ as in Theorem~\ref{thm:sobolev-trace}.

\begin{proof}[Proof of Theorem~\ref{thm:intro-lebedev-milin}]
 Recall (see~\cite{Beckner1993}) that
 \[ \oint_{\bR^n} f(-\oDelta)^{n/2}f\,dx \geq 2(n!)\Vol(S^n)\log\oint_{\bR^n} e^{f-\of}\,d\mu \]
 with equality if and only if there are constants $a\in\bR$ and $\varepsilon\in(0,\infty)$ and a point $\xi\in\bR^n$ such that
 \[ f(x) = a - \ln \frac{\varepsilon + \lv x-\xi\rv^2}{1+\lv x\rv^2} . \]
 Combining this with Corollary~\ref{cor:yang} yields the desired conclusion.
\end{proof}

\subsection*{Acknowledgments} The author thanks the anonymous referees for their helpful suggestions which clarified the exposition.

\bibliographystyle{abbrv}
\bibliography{bib}
\end{document}